\documentclass{tpms-l}

\usepackage{vfSDEMicros}

\newtheorem{theorem}{Theorem}[section]

\newtheorem{corollary}[theorem]{Corollary}
\theoremstyle{definition}

\newtheorem{proposition}[theorem]{Proposition}

\theoremstyle{remark}
\newtheorem{remark}[theorem]{Remark}

\numberwithin{equation}{section}

\begin{document}

\title[Fractional SPDE for Random Tangent Fields]{Fractional Stochastic
	Partial Differential Equation for Random Tangent Fields on the Sphere}

\author{Vo V. Anh}
\address{Faculty of Science, Engineering and Technology, Swinburne
	University of Technology, PO Box 218, Hawthorn, Victoria 3122, Australia}
\email{vanh@swin.edu.au}
\author{Andriy Olenko}
\address{Department of Mathematics and Statistics, La Trobe University,
	Melbourne, VIC, 3086, Australia}
\email{a.olenko@latrobe.edu.au}
\author{Yu Guang Wang}
\address{Max Planck Institute for Mathematics in the Sciences, Inselstrasse
	22, 04103 Leipzig, Germany}
\curraddr{Institute of Natural Sciences,  School of Mathematical Sciences, and Key Laboratory of Scientific and Engineering Computing of Ministry of Education (MOE-LSC),
	Shanghai Jiao Tong University, Shanghai 200240, China}
\email{yuguang.wang@mis.mpg.de}


\date{08/JUL/2021}


\begin{abstract}
This paper develops a fractional stochastic partial differential equation
(SPDE) to model the evolution of a random tangent vector field on the unit
sphere. The SPDE is governed by a fractional diffusion operator to model the
L\'{e}vy-type behaviour of the spatial solution, a fractional derivative in
time to depict the intermittency of its temporal solution, and is driven by
vector-valued fractional Brownian motion on the unit sphere to characterize
its temporal long-range dependence.  The solution to the SPDE is presented in
the form of the Karhunen-Lo\`{e}ve expansion in terms of vector spherical
harmonics. Its covariance matrix function is established as a tensor field
on the unit sphere that is an expansion of Legendre tensor kernels.
Approximations to the solutions are studied and convergence rates of the
approximation errors are given. It is demonstrated how these convergence
rates depend on the decay of the power spectrum and variances of the
fractional Brownian motion.
\end{abstract}

\maketitle

\section{Introduction}

Stochastic partial differential equations (SPDE) on the unit sphere $\mathbb{%
	S}^{2}$ in $\mathbb{R}^{3}$ have many applications in geophysics and
environmental modelling \cite%
{bolin2011spatial,brillinger1997particle,castruccio2013global,stein2013stochastic,stein2007spatial}%
. Various fractional versions of SPDEs have also been constructed in diverse
applications \cite%
{anh2017stochastic,anh2004riesz,anh2016fractional,beskos2015bayesian,Ovidio2014,Ovidio2016,hristopulos2003permissibility,hu2016rate,inahama2013laplace,lyons1998differential,Manikin}%
. A distinct merit of these fractional models is that they can be used to
maintain long-range dependence in the evolution of complex systems, such as
models of climate change and density fluctuations in the primordial universe
as inferred from the cosmic microwave background~(CMB) \cite{anh2018approximation,BKLO2019,BKLOO2020,castruccio2013global,Planck2016I,Planck2016IX,Planck2016XI}.

Vector fields defined on a spherical surface have also been extensively
studied. Terrestrial processes such as wind fields, oceanic currents, the
Earth's electric and magnetic fields are some of the most well-studied
examples \cite{atlas2001effects,sabaka2010mathematical,zhang2007assimilation}. Because the horizontal scale of the atmosphere and the oceans far exceeds
their vertical scale, it is natural to decompose such vector fields into
tangential and radial components and treat them separately. For example, in
space weather, the Earth's electric and magnetic fields associated with
ionospheric electric currents are distinct in the tangential and radial
directions, and many of the ionospheric electrodynamic processes can be
treated as tangent vector fields on a sphere \cite{richmond1988mapping}. In
meteorology, the wind flow at the Earth's surface level is another example
of a tangent vector field defined on a sphere \cite{pan2001vector}.

The theory of general vector temporal-spatial random
fields on compact two-point homogeneous spaces (that include the sphere) was
developed in \cite{MaMal2020}, which also includes numerous references on
the topic. However, the class of tangential random fields requires special
attention as, except for the degenerated cases, it is not a subclass of
stationary vector fields. In the general case, it requires an application of
the theory of random sections of vector and tensor bundles over $\sph{2}.$
Therefore, this paper explores an approach that is different from \cite%
{MaMal2020} as outlined below.

Modeling and analysis of vector fields (not necessarily tangential to a
sphere) is facilitated by the celebrated Helmholtz-Hodge decomposition that
expresses the vector field as the sum of a curl-free component, a
divergence-free component and a component that is the gradient of a harmonic
function \cite{FrSc2009,harouna2010helmholtz,gia2019favest,li2019fast}. The
harmonic component vanishes when the domain is a spherical surface and the
vector field is tangential to this surface. Scheuerer and Schlather~\cite%
{scheuerer2012covariance} constructed random vector fields that are either
curl-free or divergence-free; Schlather et al.~\cite{schlather2015analysis}
introduced a parametric cross-covariance model for these random vector
fields. Fan et al. \cite{fan2018modeling} constructed Gaussian tangential
vector fields on the unit sphere that are either curl-free or
divergence-free, then utilized the Helmholtz-Hodge decomposition to
construct general Gaussian tangent vector fields, called \emph{Tangent Mat%
	\'{e}rn Model}, from a pair of correlated Gaussian scalar potential fields.
For computations, a fast algorithm of vector spherical harmonic transforms
for tangent vector fields on the two-dimensional sphere was developed in
\cite{gia2019favest} and the tensor needlet approximation for spherical
tangent vector fields and its fast transform algorithm were recently
established in \cite{li2019fast}. The latter is an extension of its scalar
case \cite{le2017needlet}.

This paper will develop a fractional stochastic partial differential
equation to model the evolution of tangent vector fields on the unit sphere.
The model will take the following form:
\begin{equation}
	\pdt{t}\sol(t,\PT{x})=-\psi (-\LBo)\sol(t,\PT{x})+\IntD{\vfBmsph(t,\PT{x})}%
	,\quad t\geq 0,\;\PT{x}\in \sph{2}.  \label{eq:vfSDE.S2}
\end{equation}%
Here, $\sol(t,\PT{x})$ is a $\C^{3}$-valued random tangent (vector) field on
the sphere $\sph{2}\subseteq \Rd[3]$. The time derivative $\frac{\partial
	^{{}\beta }u}{\partial t^{{}\beta }}$ of order $\beta \in (0,1]$ is defined
as
\begin{equation*}
	\displaystyle\frac{\partial ^{\beta }}{\partial t^{\beta }}u\left(
	t,x\right) =\left\{
	\begin{array}{ll}
		\displaystyle\frac{\partial }{\partial t}u\left( t,x\right) , & \mbox{if}%
		~\beta =1, \\[3mm]
		\left( \mathcal{D}_{t}^{{}\beta }u\right) (t,x), & \mbox{if}~\beta \in (0,1),%
	\end{array}%
	\right.
\end{equation*}%
where
\begin{equation*}
	\left( \mathcal{D}_{t}^{{}\beta }u\right) \left( t,x\right) =\frac{1}{\Gamma
		\left( 1-\beta \right) }\left[ \frac{\partial }{\partial t}%
	\int_{0}^{t}\left( t-\tau \right) ^{-\beta }u\left( \tau ,x\right) d\tau -%
	\frac{u(0,x)}{t^{{}\beta }}\right] ,\quad 0<t\leq T,
\end{equation*}%
is the regularized fractional derivative or fractional derivative in the
Caputo-Djrbashian sense (see Appendix A of \cite{anh2001spectral}). The
operator $\psi (-\LBo)$ is the fractional diffusion operator which, for $%
\alpha \in (0,2]$ and $\alpha +\gamma \in \lbrack 0,2]$, is the function $%
\psi (s)=s^{\alpha /2}(1+s)^{\gamma /2}$ of the Laplace-Beltrami operator $%
\LBo$ on $\sph{2}.$ It is {an operator in a $C_{0}$-semigroup on the $L_{2}$%
	-space of spherical tangent fields} {with respect to the L\'{e}vy measure
	associated with the function $\psi $ }(to be defined in detail in Subsection
4.1). The driving noise $\vfBmsph(t,\PT{x})$ is a
tangential fractional Brownian motion on $\sph{2},$ that is defined in terms
of the vector spherical harmonics $\{(\dsh,\csh)\setsep\ell \geq 1,m=-\ell
,\dots ,\ell \}$ as
\begin{equation*}
	\vfBmsph(t,\PT{x}):=\sum_{\ell =1}^{\infty }\sum_{m=-\ell }^{\ell }\left( %
	\BMa(t)\dsh(\PT{x})+\BMb(t)\csh(\PT{x})\right) ,\;t\geq 0,\ \PT{x}\in \sph{2}%
	,
\end{equation*}%
where $\BMa(s)$ and $\BMb(s)$, $s\geq 0$, are independent fractional
Brownian motions on $\Rplus$ with Hurst index $\hurst\in \lbrack 1/2,1)$.

Using the expansion
\begin{equation*}
	\sol(t,\PT{x})=\sum_{\ell =1}^{\infty }\sum_{m=-\ell }^{\ell }\left( \hat{%
		\mathbf{X}}_{\ell m}(t)\dsh(\PT{x})+\tilde{\mathbf{X}}_{\ell m}(t)\csh(\PT{x}%
	)\right)
\end{equation*}%
of $\sol(t,\PT{x})$ in terms of the vector spherical harmonics $\{(\dsh,\csh)%
\setsep\ell \geq 1,m=-\ell ,\dots ,\ell \}$ (detailed in Section~\ref%
{sec:vsh}), the fractional differential operator and the fractional
Brownian motion in Eq.~\eqref{eq:vfSDE.S2} can be properly defined.

The solution to the equation \eqref{eq:vfSDE.S2} is then given as an
expansion in terms of vector spherical harmonics:
\begin{align}
    \sol(t,\PT{x})& =\sum_{\ell =1}^{\infty }\sum_{m=-\ell }^{\ell }\left(
    \int_{0}^{t}\left( t-s\right) ^{\beta -1}E_{\beta ,\beta }(-\left(
    t-s\right) ^{\beta }\psi (\eigvs))\IntBa(s)\dsh(\PT{x})\right.  \notag \\
    & \left. \hspace{2.7cm}+\int_{0}^{t}\left( t-s\right) ^{\beta -1}E_{\beta
        ,\beta }(-\left( t-s\right) ^{\beta }\psi (\eigvs))\IntBb(s)\csh(\PT{x}%
    )\right) ,  \notag
\end{align}%
where
\begin{equation}
	E_{a,b}(z):=\sum_{k=0}^{\infty }\frac{z^{k}}{\Gamma (ak+b)},\quad z\in \C,\
	a,b>0,  \label{eq:Mittag.Leffler.fun}
\end{equation}%
is the generalized Mittag-Leffler function with index $a$ and $b$.

Then, the solution to \eqref{eq:vfSDE.S2} under the initial condition $\sol%
(0,\PT{x})=\posol(t_{0},\PT{x})$ will be consider. Here $\posol(t_{0},\PT{x}%
) $ is the solution at $t=t_{0}$ of the following fractional stochastic
Cauchy problem:
\begin{equation}
	\pdt{t}\posol(t,\PT{x})=-\psi (-\LBo)\posol(t,\PT{x}),\quad t\geq 0,\;\PT{x}%
	\in \sph{2},  \label{eq:sCau.pb}
\end{equation}%
which is solved under the initial condition $\posol(0,\PT{x})=\rvf_{0}(\PT{x}%
)$ for a random tangent field $\rvf_{0}(\PT{x})$ on the sphere $\sph{2}$.

Eqs.~\eqref{eq:vfSDE.S2} and \eqref{eq:sCau.pb} can be used to describe the
evolution of two-stage stochastic systems. Eq.~\eqref{eq:sCau.pb} determines
the evolution on the time interval $[0,t_{0}]$ while Eq.~\eqref{eq:vfSDE.S2}
represents the solution for a system perturbed by fractional Brownian motion
on the interval $[t_{0},t_{0}+t]$. CMB is an example of such systems, as it
passed through different formation epochs, inflation, recombinatinon, etc.
\cite{Dodelson2003}.

Some background materials on vector spherical harmonics, random tangent
fields and vector-valued fractional Brownian motion are presented in
Sections~\ref{sec:vsh} and \ref{sec:random tangent fields} to prepare for
the formulation of the fractional SPDE model for random tangent fields in
Section~\ref{sec:vSDE}. Under the initial condition $\sol(t,\PT{x})=\posol%
(t_{0},\PT{x})$, the solution of \eqref{eq:vfSDE.S2} is given by
\begin{align}
    \sol(t,\PT{x})& =\sum_{\ell =1}^{\infty }\sum_{m=-\ell }^{\ell }\left\{
    \left( \dfco{(\rvf_0)}E_{\beta ,1}\bigl(-t_{0}^{\beta }\psi (\eigvs)\bigr)%
    +\int_{0}^{t}\left( t-s\right) ^{\beta -1}\right. \right.\notag \\
    &\times \left. E_{\beta ,\beta }\bigl(-\left(
    t-s\right)^{\beta } \psi (\eigvs)\bigr)\IntBa(s)\right) \dsh(\PT{x})
    +\left( \cfco{(\rvf_0)}E_{\beta ,1}\bigl(%
    -t_{0}^{\beta }\psi (\eigvs)\bigr)\right. \notag \\
    &\left.+\int_{0}^{t}\left( t-s\right) ^{\beta
        -1} \left. E_{\beta ,\beta }\bigl(-\left( t-s\right) ^{\beta }\psi (\eigvs)\bigr)%
    \IntBb(s)\right) \csh(\PT{x})\right\} ,  \label{eq:sol.vfSDE.sCaupb}
\end{align}
where $\dfco{(\rvf_0)}$ and $\cfco{(\rvf_0)}$ are the Fourier coefficients
of $\rvf_{0}(\PT{x})$ in terms of the vector spherical harmonics $\dsh$ and $%
\csh$ respectively. The covariance matrix function of %
\eqref{eq:sol.vfSDE.sCaupb} is then established as a tensor field on $\sph{2}
$ that is an expansion of Legendre tensor kernels. A{pproximations to the
	solutions and their convergence rates are investigated. It is demonstrated
	how these convergence rates depend on the decay of the power spectrum and
	variances of the fractional Brownian motion. }Section 5 derives an upper
bound in the $L_{2}\left( \Omega \times \mathbb{S}^{2}\right) $-norm of the
temporal increments of the solution to Equation (\ref{eq:vfSDE.S2}). The
result shows the interplay between the exponent $\beta $ of the fractional
derivative in time and the Hurst parameter $H$ of the fractional Brownian
motion driving the equation.{\ }

The symbol $C$ will denote constants which are not
important for our exposition. The same symbol may be used for different
constants appearing in the same proof.

\section{Vector Spherical Harmonics}

\label{sec:vsh} Let $\C^{3}$ be the $3$-dimensional complex coordinate
space. In this paper, the vectors in $\C^{3}$ are column vectors. For $\PT{x}%
\in \C^{3}$, the transpose of $\PT{x}$ is the row vector $\PT{x}^{T}$. The
inner product of two vectors $\PT{x}=(x_{1},x_{2},x_{3})^{T}$ and $\PT{y}%
=(y_{1},y_{2},y_{3})^{T}$ in $\C^{3}$ is $\PT{x}\cdot \PT{y}:=\PT{x}^{T}%
\conj{\PT{y}}:=\sum_{j=1}^{3}x_{j}\conj{y_j}$. The Euclidean norm of $\PT{x}$
is $|\PT{x}|=\sqrt{\PT{x}\cdot \PT{x}}$. The tensor product of two vectors $%
\PT{x}$ and $\PT{y}$ in $\C^{3}$ is a matrix product of $\PT{x}$ with the
conjugate transpose of $\PT{y}$: $\PT{x}\otimes \PT{y}=\PT{x}\conj{\PT{y}}%
^{T}$.

\subsection{Function Spaces and Vector Spherical Harmonics}

Let $\sph{2}$ be the unit sphere of the $3$-dimensional real coordinate
space $\Rd[3]$. Let $\Lp{2}{2}:=L_{2}(\sph{2},\sphm)$ be the $L_{2}$-space
of complex-valued functions on $\sph{2}$ with the inner product
$\InnerL[\Lp{2}{2}]{f,g}=\int_{\sph{2}}f(\PT{x})\conj{g(\PT{x})}\IntDiff{x}$
for all $f,g\in \Lp{2}{2}$, 
where $\sphm$ is the Riemann surface measure on $\sph{2}$. A $\C^{3}$-valued
function on $\sph{2}$ is called a \emph{vector field} on $\sph{2}$. For a
vector field $\PT{f}$ on $\sph{2}$ the fields
\begin{equation*}
	\PT{f}_{\mathrm{nor}}(\PT{x}):=(\PT{f}(\PT{x})\cdot \PT{x})\PT{x},\quad %
	\PT{f}_{\mathrm{tan}}(\PT{x}):=\PT{f}(\PT{x})-\PT{f}_{\mathrm{nor}}(\PT{x}%
	),\; \PT{x}\in \sph{2},
\end{equation*}%
are defined to be the \emph{normal vector field} and \emph{tangent vector
	field} of $\PT{f}$ respectively.

Let $\vLp{2}{2}$ be the $L_{2}$-space of $\C^{3}$-valued functions on the
sphere $\sph{2}$ with the inner product
\begin{equation*}
	\InnerL[\vLp{2}{2}]{\PT{f},\PT{g}}:=\int_{\sph{2}}\PT{f}(\PT{x})\cdot \PT{g}(%
	\PT{x})\IntDiff{x},\quad \PT{f},\PT{g}\in \vLp{2}{2},
\end{equation*}
and the induced $L_{2}$-norm $\normb{\PT{f}}{\vLp{2}{2}}:=\sqrt{%
	\InnerL[\vLp{2}{2}]{\PT{f},\PT{f}}}$. The $L_{2}$-space of tangent fields on
$\sph{2}$ is defined as $\tLp{2}{2}:=\bigl\{\PT{f}\in \vLp{2}{2}:\PT{f}=%
\PT{f}_{\mathrm{tan}}\bigr\}$ with the (induced) inner product
\begin{equation*}
	\InnerL{\PT{f},\PT{g}}:=\InnerL[\tLp{2}{2}]{\PT{f},\PT{g}}:=\int_{\sph{2}}%
	\PT{f}(\PT{x})\cdot \PT{g}(\PT{x})\IntDiff{x},\quad \PT{f},\PT{g}\in %
	\tLp{2}{2},
\end{equation*}%
and $L_{2}$-norm $\normb{\PT{f}}{\tLp{2}{2}}:=\sqrt{\InnerL{\PT{f},\PT{f}}}$.

This paper will focus on tangent vector fields and their temporal evolution.
For a tangent vector field $\PT{f}=(f_{1},f_{2},f_{3})^{T}\in \tLp{2}{2}$,
the integral $\int_{\sph{2}}\PT{f}(\PT{x})\IntDiff{x}$ of $\PT{f}$ on $%
\sph{2}$ is the vector
\begin{equation*}
	\left(\int_{\sph{2}}f_{1}(\PT{x})\IntDiff{x},\int_{\sph{2}}f_{2}(\PT{x})%
	\IntDiff{x},\int_{\sph{2}}f_{3}(\PT{x})\IntDiff{x}\right)^{T}.
\end{equation*}
A $\C^{3\times 3}$-valued \emph{tensor field} $\mathbf{u}$ on $\sph{2}$ is a
$\C^{3\times 3}$-valued function on $\sph{2}$. The tensor field $\PT{u}$ can
be written as $\mathbf{u}=\left( \PT{u}_{1},\PT{u}_{2},\PT{u}_{3}\right) $,
where each $\PT{u}_{i}$, $i=1,2,3$, is a vector field on $\sph{2}$.

Let $\LBo$ be the Laplace-Beltrami operator on $\sph{2}$. Let $\shY$, $\ell
\in \Nz,$ $m=-\ell ,\dots ,\ell ,$ be an orthonormal basis of complex-valued
spherical harmonics of $\Lp{2}{2}$. The functions $\shY$, $\ell \in \Nz,$ $%
m=-\ell ,\dots ,\ell,$ are eigenfunctions of $\LBo$ with eignvalues $\eigvm%
:=\ell(\ell +1)$ satisfying
\begin{equation*}
	\LBo\shY = -\eigvm\shY.
\end{equation*}
Let $\nabla_{\cdot}$ be the gradient on $\Rd[3]$. The \emph{surface-gradient}
on $\sph{2}$ is $\sfgrad:=(\sfgrad)_{\cdot}:=P_{\cdot}\nabla_{\cdot }$ with
the matrix $P_{\PT{x}}:=\imat-\PT{x}\otimes \PT{x}$, where $\imat$ is the $%
3\times 3$ identity matrix and $\PT{x}\in \sph{2}$. The \emph{surface-curl}
on $\sph{2}$ is $\sfcurl:=\sfcurl_{\cdot }:=Q_{\cdot}\nabla_{\cdot}$ with
the matrix
\begin{equation*}
	Q_{\PT{x}}:=%
	\begin{pmatrix}
		0 & -x_{3} & x_{2} \\
		x_{3} & 0 & -x_{1} \\
		-x_{2} & x_{1} & 0%
	\end{pmatrix}%
	.
\end{equation*}
Then, it holds that $\sfgrad\cdot \sfgrad=\sfcurl\cdot \sfcurl=\LBo$ (see,
e.g., \cite[Chapter~2]{FrSc2009}).

For degree $\ell \geq 1$, the \emph{divergence-free vector spherical harmonic%
} is
\begin{equation*}
	\dsh=\sfcurl\shY/\sqrt{\eigvs},
\end{equation*}%
and the \emph{curl-free vector spherical harmonic} is
\begin{equation*}
	\csh=\sfgrad\shY/\sqrt{\eigvs}.
\end{equation*}%
The set $\{(\dsh,\csh)\setsep\ell \geq 1,m=-\ell ,\dots ,\ell \}$ forms an
orthonomal basis of $\tLp{2}{2}$ (see, e.g., \cite{FuWr2009} and also \cite[%
Chapter~5]{FrSc2009}). Let $\dfco{\PT{f}}:=\InnerL{\PT{f},\dsh}$ and $%
\cfco{\PT{f}}:=\InnerL{\PT{f},\csh}$ be the \emph{divergence-free and
	curl-free Fourier coefficients} of $\PT{f}\in \tLp{2}{2}$. For $\PT{f}\in %
\tLp{2}{2}$, the representation
\begin{equation*}
	\PT{f}= \sum_{\ell =1}^{\infty }\sum_{m=-\ell }^{\ell }\left( \dfco{\PT{f}}%
	\dsh+\cfco{\PT{f}}\csh\right)
\end{equation*}%
holds in $\tLp{2}{2}$ sence, and the following Parseval identity is
satisfied
\begin{equation}
	\norm{\PT{f}}{\tLp{2}{2}}^{2}=\sum_{\ell =1}^{\infty }\sum_{m=-\ell }^{\ell
	}\left( \bigl|\dfco{\PT{f}}\bigr|^{2}+\bigl|\cfco{\PT{f}}\bigr|^{2}\right) .
	\label{eq:vf.Parseval}
\end{equation}%
For $\ell \geq 0$, let $\Legen(t)$, $t\in \lbrack -1,1],$ be the Legendre
polynomial of degree $\ell $. The divergence-free and curl-free \emph{%
	Legendre (rank-$2$) tensor kernels} of degree $\ell $ are
\begin{equation}
	\dlg(\PT{x},\PT{y})=\frac{2\ell +1}{4\pi \eigvs}\sfcurl_{\PT{x}}\otimes %
	\sfcurl_{\PT{y}}\Legen(\PT{x}\cdot \PT{y}),\;\;\clg(\PT{x},\PT{y})=\frac{%
		2\ell +1}{4\pi \eigvs}(\sfgrad)_{\PT{x}}\otimes (\sfgrad)_{\PT{y}}\Legen(%
	\PT{x}\cdot \PT{y}),\;\; \PT{x},\PT{y}\in \sph{2},  \label{eq:dlg.clg}
\end{equation}%
see \cite[Chapter~5]{FrSc2009}.

The following addition theorem holds for vector spherical harmonics and
Legendre tensor kernels (see \cite[Theorem~5.31]{FrSc2009}). For $\ell \geq
1 $ and $\PT{x},\PT{y}\in \sph{2}$,
\begin{equation}
	\sum_{m=-\ell }^{\ell }\dsh(\PT{x})\otimes \dsh(\PT{y})=\dlg(\PT{x},\PT{y}%
	),\quad \sum_{m=-\ell }^{\ell }\csh(\PT{x})\otimes \csh(\PT{y})=\clg(\PT{x},%
	\PT{y}).  \label{eq:add.thm.vsh}
\end{equation}

\section{Random Tangent Fields and Fractional Brownian Motion}

\label{sec:random tangent fields}

Let $(\probSp,\probB,\probm)$ be a probability space, where $\probB$ is the
Borel algebra on $\Omega $ with probability measure $\probm$. Let $%
\mathscr{B}(\sph{2})$ be the Borel algebra on $\sph{2}$. An $\probB\otimes %
\mathscr{B}(\sph{2})$-measurable function $T:\probSp\times \sph{2}%
\rightarrow \C$ is called a complex-valued random field on the sphere $%
\sph{2}$. For a random variable $X$ on $(\probSp,\mathcal{F},\probm),$ let $%
\expect{X}$ be the expectation of $X$. For $\C^{3}$-valued random vectors $%
\rvec,\mathbf{Y}$ on~$\probSp$, the expectation $\expect{\rvec}$ is a vector
of expectations of components of $\rvec,$ and the covariance matrix of $%
\rvec
$ and $\mathbf{Y}$ is $\cov{\rvec,\mathbf{Y}}=%
\expect{(\rvec-\expect{\rvec})\otimes
	(\mathbf{Y}-\expect{\mathbf{Y}})}$. The variance matrix of $\rvec$ is $%
\var{\rvec}=\cov{\rvec,\rvec}$. Let $\Lpprob{2}:=\Lpprob[,\probm]{2}$ be the
complex-valued $L_{2}$-space on $\probSp$ with respect to the probability
measure $\probm$, endowed with the norm $\norm{\cdot}{\Lpprob{2}}$. Let $%
\Lppsph{2}{2}:=\Lppsph[,\prodpsphm]{2}{2}$ be the complex-valued $L_{2}$%
-space on the product space of $\probSp$ and $\sph{2}$, where $\prodpsphm$
is the corresponding product measure. If it does not cause confusion, the
notation $\otimes $ will be also used as the tensor of two vectors.

\subsection{Random Tangent Fields}

$\rvf(\xi ,\PT{x}),$ $\xi\in \Omega,$ $\PT{x}\in \sph{2},$ is called a
complex-valued \emph{random tangent (vector) field} on $\sph{2}$ if the
components of $\rvf(\xi ,\PT{x})$ are complex-valued random fields on the
sphere $\sph{2}$ and $\rvf(\xi ,\PT{x})=\rvf(\xi ,\PT{x})_{\mathrm{tan}}$.
We will denote $\rvf(\xi ,\PT{x})$ by $\rvf(\PT{x})$, or $\rvf$ for brevity
if no confusion arises. For a random tangent field $\rvf$, its \emph{%
	covariance matrix function} is $\cov{\rvf(\PT{x}),\rvf(\PT{y})}$. Its \emph{%
	variance matrix function} is $\var{\rvf(\PT{x})}:=\cov{\rvf(\PT{x}),\rvf(%
	\PT{x})}$. A random tangent field $\rvf$ on $\sph{2} $ is \emph{Gaussian} if
the random vector $\rvf(\PT{x})$ follows a $3$-variate normal distribution
for each $\PT{x}\in \sph{2}$.

Let $\vLppsph{2}{2}:=\vLppsph[,\prodpsphm]{2}{2}$ be the $L_{2}$-space of
random tangent fields with the product measure $\prodpsphm$, endowed with
inner product
\begin{equation*}
	{\InnerL{\PT{f},\PT{g}}}_{\vLppsph{2}{2}}:=\int_{\prodpsph}\PT{f}(\xi ,\PT{x}%
	)\cdot \PT{g}(\xi ,\PT{x})\IntD{\prodpsphm}(\xi ,\PT{x}),
\end{equation*}%
and induced $L_{2}$-norm $\Vert \PT{f}\Vert _{\vLppsph{2}{2}}:=\sqrt{{%
		\InnerL{\PT{f},\PT{f}}}_{\vLppsph{2}{2}}}$ for $\PT{f},\PT{g}\in %
\vLppsph{2}{2}$.

By Fubini's theorem,
\begin{equation*}
	{\InnerL{\PT{f},\PT{g}}}_{\vLppsph{2}{2}}= \expect{\int_{\sph{2}}\PT{f}(%
		\cdot,\PT{x})\conj{\PT{g}(\cdot,\PT{x})}^{T}\IntDiff{x}}= %
	\expect{\InnerL[\tLp{2}{2}]{\PT{f},\PT{g}}}.
\end{equation*}
In particular,
\begin{equation}  \label{eq:rvf.L2}
	\|\PT{f}\|_{\vLppsph{2}{2}}^{2} = \expect{\|\PT{f}\|_{\tLp{2}{2}}^{2}},
\end{equation}
which implies $\PT{f}(\xi,\cdot)\in\tLp{2}{2}$ $\Pas$.

\subsection{Vector-valued Fractional Brownian Motion}

Let $\hurst\in \lbrack 1/2,1)$ and $\sigma >0$. A real-valued fractional
Brownian motion $\rfBm(t)$, $t\geq 0,$ with the Hurst index $\hurst$ and
variance $\sigma $ at $t=1$ is a (real-valued) centred Gaussian process
satisfying
\begin{equation*}
	\rfBm(0)=0,\quad \expect{\bigl|\rfBm(t)-\rfBm(s)\bigr|^{2}}=|t-s|^{2\hurst%
	}\sigma ^{2},
\end{equation*}%
see, for example, \cite{BiHuOkZh2008}.

A complex-valued Brownian motion $\fBm(t)$, $t\ge0$ with the Hurst index $%
\hurst$ and variance $\sigma$ at $t=1$ is given by
\begin{equation*}
	\fBm(t) = \rfBm[1](t) + \imu \rfBm[2](t),
\end{equation*}
where $\rfBm[1](t)$ and $\rfBm[2](t)$ are two independent real-valued
fractional Brownian motions with the Hurst index $\hurst$ and variance $%
\sigma$ at $t=1$.

For $\hurst\in \lbrack 1/2,1)$, let $\BMa(t)$ and $\BMb(t)$ be independent
centred complex-valued fractional Brownian motions with the Hurst index $%
\hurst$ and variances $A_{\ell }^{1}$ and $A_{\ell }^{2}$ at $t=1$
satisfying
\begin{equation}
	\sum_{\ell =0}^{\infty }(2\ell +1)A_{\ell }^{1}<\infty\quad \mbox{and} \quad
	\sum_{\ell =0}^{\infty }(2\ell +1)A_{\ell }^{2}<\infty  \label{brspect}
\end{equation}
respectively. Following \cite[Definition~2.1]{GrAn1999} and also \cite%
{anh2018approximation}, we define an $\tLp{2}{2}$-valued fractional Brownian
motion $\vfBmsph(t,\PT{x})$ as an expansion with respect to vector spherical
harmonics with $\BMa(t)$ and $\BMb(t)$ as coefficients, i.e.,
\begin{equation}
	\vfBmsph(t,\PT{x}):=\sum_{\ell =1}^{\infty }\sum_{m=-\ell }^{\ell }\left( %
	\BMa(t)\dsh(\PT{x})+\BMb(t)\csh(\PT{x})\right), \; t\geq 0,\ \PT{x}\in %
	\sph{2}.  \label{eq:fBm}
\end{equation}

For a measurable function $g$ on $\Rplus$, the stochastic integral of $g$
with respect to a complex-valued fractional Brownian motion $\fBm(t)$ is
\begin{equation*}
	\int_{0}^{t}g(s)\IntD{\fBm(s)}:=\int_{0}^{t}g(s)\IntD{\rfBm[1](s)}+\imu%
	\int_{0}^{t}g(s)\IntD{\rfBm[2](s)},
\end{equation*}
where each integral $\int_{0}^{t}g(s)\IntD{\rfBm[i](s)}$, $i=1,2$, is
defined as a Riemann-Stietjes integral.

\section{Fractional Stochastic Partial Differential Equations for Tangent
	Fields}

\label{sec:vSDE}

This section presents the solution to Eq.~\eqref{eq:vfSDE.S2} given as an
expansion with respect to vector spherical harmonics. Before solving the
equation, a proper definition of the fractional diffusion operator $\psi (-%
\LBo)$ is presented.

\subsection{Fractional Diffusion Operator for Random Tangent Fields}

{The fractional diffusion operator for vector fields is defined as an
	operator in a $C_{0}$-semigroup on the $L_{2}$-space of spherical tangent
	fields $\tLp{2}{2}$ as follows. By \cite[Section~4.1]{AnMc2004}, for $\alpha
	+\gamma \in (0,2]$, there exists a L\'{e}vy measure $\Levym{t}$ on $%
	[0,\infty )$ such that
	\begin{equation}
		\psi (z)=\int_{0}^{\infty }\left( 1-e^{-tz}\right) \Levym{t}.
		\label{eq:psi.Levy.measure}
	\end{equation}%
	For $t\geq 0$, we define the operator
	\begin{equation}
		\mathrm{P}_{t}\PT{f}:=\sum_{\ell =1}^{\infty }\sum_{m=-\ell }^{\ell }\left( %
		\dfco{\PT{f}}\dsh+\cfco{\PT{f}}\csh\right) e^{-t\eigvm},\quad \PT{f}\in %
		\tLp{2}{2}.  \label{eq:Pt.f}
	\end{equation}%
	Then $\{\mathrm{P}_{t}\}_{t\geq 0}$ forms a $C_{0}$-semigroup for the $L_{2}$
	space of spherical tangent fields $\tLp{2}{2}$. We define the operator $%
	\mathrm{D}^{\psi }$ using the semigroup $\{\mathrm{P}_{t}\}_{t\geq 0}$ and L%
	\'{e}vy measure $\Levym{t}$ associated with $\psi $ by
	\begin{equation*}
		\mathrm{D}^{\psi }\PT{f}:=\int_{0}^{\infty }(\mathrm{P}_{t}\PT{f}-\PT{f})%
		\Levym{t},\quad \PT{f}\in \tLp{2}{2}.
	\end{equation*}%
	By the orthonormality of $\dsh$ and $\csh$ and \eqref{eq:psi.Levy.measure}
	and \eqref{eq:Pt.f}, for $\ell \geq 1$ and $m=-\ell ,\dots ,\ell $,
	\begin{equation*}
		\mathrm{D}^{\psi }\dsh=\int_{0}^{\infty }\dsh(e^{-t\eigvs}-1)\Levym{t}=-\psi
		(\eigvm)\dsh,
	\end{equation*}%
	and similarly,
	\begin{equation*}
		\mathrm{D}^{\psi }\csh=-\psi (\eigvm)\csh.
	\end{equation*}%
	We can thus write
	\begin{equation}
		\psi (-\LBo)\PT{f}:=-\mathrm{D}^{\psi }\PT{f}:=-\int_{0}^{\infty }(\mathrm{P}%
		_{t}\PT{f}-\PT{f})\Levym{t},\quad \PT{f}\in \tLp{2}{2},  \label{eq:fDiffu}
	\end{equation}%
	and call $\psi (-\LBo)$ \emph{fractional diffusion operator} for tangent
	vector fields on $\sph{2}$. For $\rvf\in \vLppsph{2}{2}$, $\rvf(\xi ,\cdot
	)\in \tLp{2}{2}$ $\Pas$, then, one} can define $\psi (-\LBo)\rvf$ for a
random tangent field $\rvf\in \vLppsph{2}{2}$ by $\psi (-\LBo)\rvf(\xi
,\cdot )$ as given by \eqref{eq:fDiffu} for $\xi \in \probSp$.

\begin{remark}
	The expansion of the operator $D^{\psi }\PT{f}\left( \mathbf{x}\right) $ in terms
	of the vector spherical harmonics $\{\mathbf{y}_{l,m},\mathbf{z}_{l,m}):\ell
	\geq 1,m=-\ell ,...,\ell \}$ involves the Laplace exponent of a L\'{e}vy
	subordinator in its coefficients. In fact, from the L\'{e}vy measure $\nu
	\left( dt\right) $ defined by the equation%
	\begin{equation*}
		z^{\alpha /2}\left( 1+z\right) ^{\gamma /2}=\int_{0}^{\infty }\left(
		1-e^{-tz}\right) \nu \left( dt\right) ,
	\end{equation*}%
	we define a L\'{e}vy subordinator (i.e. a non-decreasing L\'{e}vy process) $%
	Z\left( t\right) $ with characteristic function%
	\begin{equation*}
		\mathrm{E}e^{i\xi Z\left( t\right) }=e^{-t\int_{0}^{\infty }\left( e^{i\xi
				\lambda }-1\right) \nu \left( d\lambda \right) }.
	\end{equation*}%
	The Laplace transform of the distribution of this subordinator then can be
	written as%
	\begin{equation*}
		\mathrm{E}e^{-t\Psi \left( \xi \right) }=e^{-t\int_{0}^{\infty }\left(
			1-e^{-\xi \lambda }\right) \nu \left( d\lambda \right) },
	\end{equation*}%
	where $\Psi \left( \xi \right) $ is known as the Laplace exponent or the
	symbol of the subordinator. As defined above,%
	\begin{align*}
		\mathrm{D}^{\psi }\PT{f}\left( \mathbf{x}\right) &=\int_{0}^{\infty}\left(
		P_{t}\PT{f}\left( \mathbf{x}\right) -P_{0}\PT{f}\left( \mathbf{x}\right) \right) \nu
		\left( dt\right)\\
		&=\sum_{l=0}^{\infty }\sum_{m=-l}^{l}\left(\dfco{\PT{f}}\mathbf{y}%
		_{l,m}\left( \mathbf{x}\right) +\cfco{\PT{f}}\mathbf{z}_{l,m}\left( \mathbf{x}%
		\right) \right) \int_{0}^{\infty }\left( e^{-t\lambda _{l}}-1\right) \nu
		\left( dt\right)\\
		&=-\sum_{l=0}^{\infty }\sum_{m=-l}^{l}\left( \dfco{\PT{f}}\mathbf{y}%
		_{l,m}\left( \mathbf{x}\right) +\cfco{\PT{f}}\mathbf{z}_{l,m}\left( \mathbf{x}%
		\right) \right) \Psi \left( \lambda _{l}\right) .
	\end{align*}%
	This result indicates that fractional diffusion operators can also be
	generated by using symbols of L\'{e}vy subordinators.
\end{remark}

\begin{remark}
	Of the same form as $-\psi (-\LBo),$ the  operator $%
	\mathcal{-}\left( -\Delta \right) ^{\alpha /2}\left( I-\Delta \right)
	^{\gamma /2}$ on $\mathbb{R}^{n}$ was investigated in \cite{anh2004riesz}.
	The fractional diffusion operator $\mathcal{-}\left( -\Delta \right) ^{\alpha /2}\left( I-\Delta
	\right) ^{\gamma /2},$ which is the inverse of the composition of the Riesz
	potential $\left( -\Delta \right) ^{-\alpha /2},$ $\alpha \in \left( 0,2%
	\right] ,$ defined by the kernel
	\begin{equation*}
		J_{\alpha }\left( x\right) =\frac{\Gamma \left( n/2-\alpha \right) }{\pi
			^{n/2}4^{\alpha }\Gamma \left( \alpha \right) }\left\vert x\right\vert
		^{2\alpha -n},\quad x\in \mathbb{R}^{n},
	\end{equation*}%
	and the Bessel potential $\left( I-\Delta \right) ^{-\gamma /2},\gamma \geq
	0,\ $defined by the kernel (see \cite{stein1970singular})
	\begin{equation*}
		I_{\gamma }\left( x\right) =\left[ \left( 4\pi \right) ^{\gamma }\Gamma
		\left( \gamma \right) \right] ^{-1}\int_{0}^{\infty }e^{-\pi \left\vert
			x\right\vert ^{2}/s}e^{-s/4\pi }s^{\left( -n/2+\gamma \right) }\frac{ds}{s}%
		,\quad x\in \mathbb{R}^{n},
	\end{equation*}%
	is the infinitesimal generator of a strongly continuous bounded holomorphic
	semigroup of angle $\pi /2$ on $L^{p}\left( \mathbb{R}^{n}\right) $ for $%
	\alpha >0,\alpha +\gamma \geq 0$ and any $p\geq 1$ (\cite{anh2004riesz}).
	This semigroup defines the Riesz-Bessel distribution (and the resulting
	Riesz-Bessel motion) if and only if $\alpha \in (0,2],$ $\alpha +\gamma \in
	\lbrack 0,2]$. When $\gamma =0$, the fractional Laplacian $-\left( -\Delta
	\right) ^{\alpha /2},$ $\alpha \in \left( 0,2\right] $, generates the L\'{e}%
	vy $\alpha $-stable distribution. While the exponent of the inverse of the
	Riesz potential indicates how often large jumps occur, it is the combined
	effect of the inverses of the Riesz and Bessel potentials that describes
	various non-Gaussian behaviours of the process \cite%
	{anh2004riesz,anh2017stochastic}. For example, depending on the sum $\alpha
	+\gamma $ of the exponents of the inverses of the Riesz and Bessel
	potentials, the Riesz-Bessel motion will be either a compound Poisson
	process, a pure jump process with jumping times dense in $[0,\infty )$ or
	the sum of a compound Poisson process and an independent Brownian motion.
\end{remark}

\subsection{Fractional SPDE and Solution}

Using the fractional diffusion operator and vector-valued fractional
Brownian motion as defined above, the fractional stochastic equation %
\eqref{eq:vfSDE.S2} is well-defined. The solution of \eqref{eq:vfSDE.S2}, as
shown by the following theorem, can be written as an expansion in terms of
vector spherical harmonics.

\begin{theorem}
    \label{thm:sol.vfSDE.S2} Let $\alpha \in (0,2]$, $\alpha +\gamma \in \lbrack
    0,2]$, $\beta \in (0,1]$ and $\hurst\in \lbrack 1/2,1)$. Let Eq.~%
    \eqref{eq:vfSDE.S2} be defined with the fractional diffusion operator %
    \eqref{eq:fDiffu}, fractional Brownian motion \eqref{eq:fBm}, Hurst index $%
    \hurst$ and variances $A_{\ell }^{1}$ and and $A_{\ell }^{2}$ at $t=1$
    satisfying
    \begin{equation*}
        \sum_{\ell \geq 1}(2\ell +1)\left( \psi (\eigvs)\right) ^{\tau }A_{\ell
        }^{1}<\infty \quad \mbox{and}\quad \sum_{\ell \geq 1}(2\ell +1)\left( \psi (%
        \eigvs)\right) ^{\tau }A_{\ell }^{2}<\infty
    \end{equation*}%
    respectively for $\tau :=\max \{\frac{2}{\beta }(1-\beta -\hurst),0\}$.
    Then, the solution to \eqref{eq:vfSDE.S2} is given by
    \begin{align}
        \sol(t,\PT{x})& =\sum_{\ell =1}^{\infty }\sum_{m=-\ell }^{\ell }\left(
        \int_{0}^{t}\left( t-s\right) ^{\beta -1}E_{\beta ,\beta }(-\left(
        t-s\right) ^{\beta }\psi (\eigvs))\IntBa(s)\dsh(\PT{x})\right.  \notag \\
        & \left. \hspace{2.7cm}+\int_{0}^{t}\left( t-s\right) ^{\beta -1}E_{\beta
            ,\beta }(-\left( t-s\right) ^{\beta }\psi (\eigvs))\IntBb(s)\csh(\PT{x}%
        )\right) ,  \label{eq:sol.vfSDE.S2}
    \end{align}%
    where the above expansion is convergent in $\vLppsph{2}{2}$.
\end{theorem}

\begin{proof}
    Let $\delta (t)$ be the Dirac delta function and $\lambda \in \R$. By \cite[%
    Chapter~5]{Podlubny1999}, for $\beta \in (0,1]$, the solution to the
    fractional equation
    \begin{equation*}
        \frac{\mathrm{d}^{\beta }f(t)}{\mathrm{d}{t}^{\beta }}+\lambda f(t)=\delta
        (t)
    \end{equation*}%
    is
    \begin{equation}
        G_{\lambda }(t)=t^{\beta -1}E_{\beta ,\beta }(-\lambda t^{\beta }),\quad
        t\geq 0,  \label{eq:G.lambda}
    \end{equation}%
    where $E_{\beta ,\beta }(\cdot )$ is the generalized Mittag-Leffler function
    given by \eqref{eq:Mittag.Leffler.fun}. Using this result and taking the
    fractional-in-time derivatives of both sides of Eq.~\eqref{eq:sol.vfSDE.S2}
    one obtains
    \begin{align*}
        \pdt{t}\sol(t,\PT{x})& =\pdt{t}\sum_{\ell =1}^{\infty }\sum_{m=-\ell }^{\ell
        }\left( \int_{0}^{t}G_{\psi (\eigvs)}(t-s)\IntBa(s)\dsh(\PT{x})\right. \\
        & \left. \hspace{1.7cm}+\int_{0}^{t}G_{\psi (\eigvs)}(t-s)\IntBb(s)\csh(%
        \PT{x})\right) \\
        & =\sum_{\ell =1}^{\infty }\sum_{m=-\ell }^{\ell }\left( \int_{0}^{t}\left(
        -\psi (\eigvs)G_{\psi (\eigvs)}(t-s)+\delta (t-s)\right) \IntBa(s)\dsh(\PT{x}%
        )\right. \\
        & \left. \hspace{1.7cm}+\int_{0}^{t}\left( -\psi (\eigvs)G_{\psi (\eigvs%
            )}(t-s)+\delta (t-s)\right) \IntBb(s)\csh(\PT{x})\right) \\
        & =\sum_{\ell =1}^{\infty }\sum_{m=-\ell }^{\ell }\left\{ \left(
        \int_{0}^{t}\left( -\psi (\eigvs)G_{\psi (\eigvs)}(t-s)\right) \IntBa(s)+%
        \IntBa(t)\right) \dsh(\PT{x})\right. \\
        & \left. \hspace{1.7cm}+\left( \int_{0}^{t}\left( -\psi (\eigvs)G_{\psi (%
            \eigvs)}(t-s)\right) \IntBb(s)+\IntBb(t)\right) \csh(\PT{x})\right\} ,
    \end{align*}%
    where the last equality uses the the property of convolution of the Dirac
    delta function. On the other hand, the right-hand side of \eqref{eq:vfSDE.S2}
    is
    \begin{align*}
        & -\psi (-\LBo)\sol(t,\PT{x})+\IntD{\vfBmsph(t,\PT{x})} \\
        & \qquad =-\sum_{\ell =1}^{\infty }\sum_{m=-\ell }^{\ell }\left(
        \int_{0}^{t}G_{\psi (\eigvs)}(t-s)\IntBa(s)\psi (-\LBo)\dsh(\PT{x})\right. \\
        & \qquad \left. \hspace{2.7cm}+\int_{0}^{t}G_{\psi (\eigvs)}(t-s)\IntBb%
        (s)\psi (-\LBo)\csh(\PT{x})\right) +\IntD{\vfBmsph(t,\PT{x})} \\
        & \qquad =\sum_{\ell =1}^{\infty }\sum_{m=-\ell }^{\ell }\left\{ \left(
        \int_{0}^{t}\left( -\psi (\eigvs)G_{\psi (\eigvs)}(t-s)\right) \IntBa(s)+%
        \IntBa(t)\right) \dsh(\PT{x})\right. \\
        & \qquad \left. \hspace{2.7cm}+\left( \int_{0}^{t}\left( -\psi (\eigvs%
        )G_{\psi (\eigvs)}(t-s)\right) \IntBb(s)+\IntBb(t)\right) \csh(\PT{x}%
        )\right\} ,
    \end{align*}%
    where the last equality uses the expansion of $\vfBmsph(t,\PT{x})$ in %
    \eqref{eq:fBm}. Thus, the random field defined by \eqref{eq:sol.vfSDE.S2}
    satisfies Eq. \eqref{eq:vfSDE.S2}.

    We now prove that the expansion \eqref{eq:sol.vfSDE.S2} converges in $%
    \vLppsph{2}{2}$. By \eqref{eq:rvf.L2} and \eqref{eq:vf.Parseval} and \cite[%
    Theorem~1.1]{MeMiVa2001}, there exist constants $C_{\hurst}^{1}$ and $C_{%
        \hurst}^{2}$ such that
    \begin{align}
        \norm{\sol(t,\PT{x})}{\vLppsph{2}{2}}^{2}& =\sum_{\ell =1}^{\infty
        }\sum_{m=-\ell }^{\ell }\left\{ \mathbb{E}\left[ \left\vert
        \int_{0}^{t}\left( t-s\right) ^{\beta -1}E_{\beta ,\beta }(-\left(
        t-s\right) ^{\beta }\psi (\eigvs))\IntBa(s)\dsh(\PT{x})\right\vert ^{2}%
        \right] \right.  \notag \\
        & \hspace{1.cm}\left. +\mathbb{E}\left[ \left\vert \int_{0}^{t}\left(
        t-s\right) ^{\beta -1}E_{\beta ,\beta }(-\left( t-s\right) ^{\beta }\psi (%
        \eigvs))\IntBb(s)\dsh(\PT{x})\right\vert ^{2}\right] \right\}  \notag \\
        &\hspace{-1.5cm} \leq \sum_{\ell =1}^{\infty }\sum_{m=-\ell }^{\ell }\left\{ C_{\hurst%
        }^{1}A_{\ell }^{1}\left( \int_{0}^{t}\bigl|\left( t-s\right) ^{\beta
            -1}E_{\beta ,\beta }(-\left( t-s\right) ^{\beta }\psi (\eigvs))\bigr|^{\frac{%
                1}{\hurst}}\IntD{s}\right) ^{2\hurst}\right.  \notag \\
        & \hspace{-1.5cm}\quad\quad\left. +C_{\hurst}^{2}A_{\ell }^{2}\left( \int_{0}^{t}\bigl|%
        \left( t-s\right) ^{\beta -1}E_{\beta ,\beta }(-\left( t-s\right) ^{\beta
        }\psi (\eigvs))\bigr|^{\frac{1}{\hurst}}\IntD{s}\right) ^{2\hurst}\right\}
        \notag \\
        & \hspace{-1.5cm}=\sum_{\ell =1}^{\infty }\sum_{m=-\ell }^{\ell }\left\{ C_{\hurst%
        }^{1}A_{\ell }^{1}\left( \int_{0}^{t}\bigl|s^{\beta -1}E_{\beta ,\beta
        }(-s^{\beta }\psi (\eigvs))\bigr|^{\frac{1}{\hurst}}\IntD{s}\right) ^{2\hurst%
        }\right.  \notag \\
        &\hspace{-1.5cm}\quad\quad\left. +C_{\hurst}^{2}A_{\ell }^{2}\left( \int_{0}^{t}\bigl|s^{\beta
            -1}E_{\beta ,\beta }(-s^{\beta }\psi (\eigvs))\bigr|^{\frac{1}{\hurst}}%
        \IntD{s}\right) ^{2\hurst}\right\}  \notag \\
        & \hspace{-1.5cm}=C_{\hurst}^{1}\sum_{\ell =1}^{\infty }(2\ell +1)A_{\ell }^{1}\left( \psi (%
        \eigvs)\right) ^{\frac{2}{\beta }(1-\beta -\hurst)}\left( \int_{0}^{t\left(
            \psi (\eigvs)\right) ^{\frac{1}{\beta }}}|u^{\beta -1}E_{\beta ,\beta
        }(-u^{\beta })|^{\frac{1}{\hurst}}\IntD{u}\right) ^{2\hurst}  \notag \\
        & \hspace{-1.5cm}\quad\quad +C_{\hurst}^{2}\sum_{\ell =1}^{\infty }(2\ell +1)A_{\ell }^{2}\left(
        \psi (\eigvs)\right) ^{\frac{2}{\beta }(1-\beta -\hurst)}\left(
        \int_{0}^{t\left( \psi (\eigvs)\right) ^{\frac{1}{\beta }}}|u^{\beta
            -1}E_{\beta ,\beta }(-u^{\beta })|^{\frac{1}{\hurst}}\IntD{u}\right) ^{2%
            \hurst}  \notag \\
        & \hspace{-1.5cm}\leq C_{\hurst}^{1}\sum_{\ell =1}^{\infty }(2\ell +1)A_{\ell }^{1}\left(
        \psi (\eigvs)\right) ^{\frac{2}{\beta }(1-\beta -\hurst)}\left(
        \int_{0}^{\infty }|G_{1}(u)|^{\frac{1}{\hurst}}\IntD{u}\right) ^{2\hurst}
        \notag \\
        &\hspace{-1.5cm}\quad\quad +C_{\hurst}^{2}\sum_{\ell =1}^{\infty }(2\ell +1)A_{\ell
        }^{2}\left( \psi (\eigvs)\right) ^{\frac{2}{\beta }(1-\beta -\hurst)}\left(
        \int_{0}^{\infty }|G_{1}(u)|^{\frac{1}{\hurst}}\IntD{u}\right) ^{2\hurst%
        }<\infty ,  \label{proofTh4.1}
    \end{align}%
    where the last inequality uses the property that for $\hurst\in \lbrack
    1/2,1)$ the function $G_{1}(t)$ defined by~\eqref{eq:G.lambda} is in $L_{{1}/%
        {\hurst}}(\Rplus)$ (see \cite[Theorem~1.3-3]{Djrbashian1993}). This means
    that the expansion in \eqref{eq:sol.vfSDE.S2} converges in $\vLppsph{2}{2}$
    and $\sol(t,\PT{x})$ defined by \eqref{eq:sol.vfSDE.S2} is the solution to
    the vector-valued fractional SPDE \eqref{eq:vfSDE.S2}, which completes the
    proof.\hfill {}
\end{proof}

Let us consider the approximation $\sol_{L}(t,\PT{x})$ of truncation degree $%
L\in \mathbb{N}$ to the solution $\sol(t,\PT{x})$ in~\eqref{eq:sol.vfSDE.S2}
defined by
\begin{align}
    \sol_{L}(t,\PT{x})& =\sum_{\ell =1}^{L}\sum_{m=-\ell }^{\ell }\left(
    \int_{0}^{t}\left( t-s\right) ^{\beta -1}E_{\beta ,\beta }(-\left(
    t-s\right) ^{\beta }\psi (\eigvs))\IntBa(s)\dsh(\PT{x})\right.  \notag \\
    & \left. \hspace{2.cm}+\int_{0}^{t}\left( t-s\right) ^{\beta -1}E_{\beta
        ,\beta }(-\left( t-s\right) ^{\beta }\psi (\eigvs))\IntBb(s)\csh(\PT{x}%
    )\right) ,\quad t>0.  \notag
\end{align}%
The following result gives the convergence rate of the approximation $\sol%
_{L}(t,\PT{x})$ to the solution $\sol(t,\PT{x})$. It demonstrates how the
convergence rate is determined by the magnitudes of variances $A_{\ell }^{1}$
and and $A_{\ell }^{2}$. {We also give its asymptotic for an important
    specific scenario of the algebraic decay. }

\begin{corollary}
    \label{cor1} Let the conditions of Theorem~\ref{thm:sol.vfSDE.S2} be
    satisfied and $\sol(t,\PT{x})$ be the solution to the equation %
    \eqref{eq:vfSDE.S2}. Then,

    \begin{enumerate}
        \item[(i)] For $t>0$, the truncation error is bounded by
        \begin{equation*}
            \bigl\|\sol(t,\PT{x})-\sol_L(t,\PT{x})\bigr\|_{\vLppsph{2}{2}} \le
            C\left(\sum_{\ell
                =L+1}^{\infty}(2\ell+1)\left(A_{\ell}^{1}+A_{\ell}^{2}\right)\left(\psi(%
            \eigvs)\right)^{\tau}\right)^{1/2}.
        \end{equation*}

        \item[(ii)] For the magnitudes of variances that decay algebraically with
        order $\nu >2$, that is, there exist constants $C >0$ and $\ell_0\in \mathbb{%
            N}$ such that $\left(A_{\ell}^{1}+A_{\ell}^{2}\right)\left(\psi(\eigvs%
        )\right)^{\tau}\le C\cdot \ell^{-\nu}$ for all $\ell\ge \ell_0$, it holds
        \begin{equation*}
            \bigl\|\sol(t,\PT{x})-\sol_L(t,\PT{x})\bigr\|_{\vLppsph{2}{2}} \le C {L}^{-%
                \frac{\nu-2}{2}};
        \end{equation*}

        \item[(iii)] The estimates in (i) and (ii) are also valid uniformly over $%
        \PT{x}\in \sph{2}$ for the root mean squared truncation error
        \begin{align*}
            \mathbf{RMSE}(\sol(t,\PT{x})-\sol_L(t,\PT{x})) &=\mathbf{Var}^{1/2}\left(\sol%
            (t,\PT{x}) - \sol_L(t,\PT{x})\right) \\
            &=\bigl\|\sol(t,\PT{x})-\sol_L(t,\PT{x})\bigr\|_{L_2(\Omega)}.
        \end{align*}

        \item[(iv)] Under conditions of the algebraic decay in (ii), for any $%
        \varepsilon>0$ it holds
        \begin{equation*}
            \mathbf{P}\Big(|\sol(t,\PT{x})-\sol_L(t,\PT{x})|\ge\varepsilon\Big)\le \frac{%
                C}{{L}^{\nu-2}\,\varepsilon^2}.
        \end{equation*}
    \end{enumerate}
\end{corollary}

\begin{proof}
    Noting that
    \begin{align}
        \sol(t,\PT{x})-\sol_{L}(t,\PT{x})& =\sum_{\ell =L+1}^{\infty }\sum_{m=-\ell
        }^{\ell }\left( \int_{0}^{t}\left( t-s\right) ^{\beta -1}E_{\beta ,\beta
        }(-\left( t-s\right) ^{\beta }\psi (\eigvs))\IntBa(s)\dsh(\PT{x})\right.
        \notag \\
        & \left. +\int_{0}^{t}\left( t-s\right) ^{\beta -1}E_{\beta
            ,\beta }(-\left( t-s\right) ^{\beta }\psi (\eigvs))\IntBb(s)\csh(\PT{x}%
        )\right) ,\quad t>0,  \notag
    \end{align}%
    the proof of statement (i) is analogous to the derivations \eqref{proofTh4.1}
    in the proof of Theorem~\ref{thm:sol.vfSDE.S2} with
    \begin{equation*}
        C=\max \left( \sqrt{C_{\hurst}^{1}},\sqrt{C_{\hurst}^{2}}\right) \left(
        \int_{0}^{\infty }|G_{1}(u)|^{\frac{1}{\hurst}}\IntD{u}\right) ^{\hurst%
        }<\infty .
    \end{equation*}

    The statement (ii) follows from (i) and the estimate
    \begin{equation*}
        \sum_{\ell
            =L+1}^{\infty}(2\ell+1)\left(A_{\ell}^{1}+A_{\ell}^{2}\right)\left(\psi (%
        \eigvs)\right)^{\tau}\le C\sum_{l=L+1}^{\infty}l^{-(\nu-1)}\le C
        L^{-(\nu-2)}.
    \end{equation*}

    Note that the proofs of \eqref{proofTh4.1} in Theorem~\ref{thm:sol.vfSDE.S2}
    and statements (i) and (ii) of this Corollary do not depend on $\PT{x}$.
    Therefore, they also uniformly hold in the $L_{2}(\Omega )$-norm over the
    sphere $\sph{2}$.

    Applying the estimate in (ii) and Chebyshev's inequality, one gets the upper
    bound in (iv).\hfill {}
\end{proof}

In the following, to introduce the initial conditions for the stochastic
Cauchy problem (\ref{eq:sCau.pb}) we will use a centred random tangent field
$\rvf_{0}(\PT{x})$. Note that, in the classical scalar and vector cases
initial conditions are usually given by stationary random fields. However,
the class of stationary tangent random fields is degenerate and consists of
almost sure constant fields. To provide more comprehensive initial
conditions one needs a non-degenerate equivalent of the stationarity
concept for tangent random fields. We introduce it via properties of the
Fourier coefficients of $\rvf_{0}(\PT{x})$, which coincides with
corresponding properties of spectral coefficients of scalar and vector
stationary fields.

Namely, let us consider a series representation of $\rvf_{0}(\PT{x})$ in
terms of the vector spherical harmonics $\dsh$ and $\csh$. The Fourier
coefficients of $\rvf_{0}(\PT{x})$ in this representation will be denoted by
$\dfco{(\rvf_0)}$ and $\cfco{(\rvf_0)}$ respectively and can be computed as
\begin{equation*}
	\dfco{(\rvf_0)}=\InnerL{\rvf_{0}(\PT{x}),\dsh}\quad \mbox{and} \quad %
	\cfco{(\rvf_0)}=\InnerL{\rvf_{0}(\PT{x}),\csh}.
\end{equation*}
We assume that these coefficients are uncorrelated, i.e. for $\ell,
\ell^{\prime}\geq 1$, $m=-\ell,\dots,\ell$, $m^{\prime}=-\ell^{\prime},%
\dots,\ell^{\prime}$ it holds that
\begin{equation}  \label{norm1}
	\expect{\dfco{(\rvf_0)}\cfco[\ell' m']{(\rvf_0)}}=0,
\end{equation}
\begin{equation}  \label{norm2}
	\expect{\dfco{(\rvf_0)}\dfco[\ell' m']{(\rvf_0)}}=\delta_{\ell
		\ell^{\prime}}\delta_{mm^{\prime}}\widehat{\sigma}_{\ell}^{2} \quad %
	\mbox{and} \quad \expect{\cfco{(\rvf_0)}\cfco[\ell' m']{(\rvf_0)}} =
	\delta_{\ell \ell^{\prime}}\delta_{mm^{\prime}}\widetilde{\sigma}_{\ell}^{2},
\end{equation}
where $\delta_{\ell \ell^{\prime}}$ is the Kronecker delta function.

The coefficients $\{\widehat{\sigma}_{\ell}^{2}\setsep\ell \geq 1\}$ and $\{%
\widetilde{\sigma}_{\ell}^{2}\setsep\ell \geq 1\}$ are divergence-free and
curl-free power spectra of $\rvf_{0}(\PT{x})$ satisfying
\begin{equation*}
	\sum_{\ell \geq 1}(2\ell +1)\widehat{\sigma}_{\ell}^{2}<\infty \quad %
	\mbox{and} \quad \sum_{\ell\geq1}(2\ell+1)\widetilde{\sigma}%
	_{\ell}^{2}<\infty
\end{equation*}
respectively.

\begin{theorem}
	\label{thm:posol} Let $\beta \in (0,1]$ and $\rvf_0(\xi,\PT{x})$ be a random
	tangent field in $\vLppsph{2}{2}$ satisfying \eqref{norm1} and \eqref{norm2}%
	. Then,  for $t\geq 0$ and $\PT{x}\in \sph{2}$ the solution of the fractional stochastic Cauchy problem %
	\eqref{eq:sCau.pb} with the initial condition $\posol(0,\PT{x})=\rvf_{0}(%
	\PT{x})$ is given in terms of vector spherical harmonics as
	\begin{equation}
		\posol(t,\PT{x}) =\sum_{\ell=1}^{\infty}\sum_{m=-\ell}^{\ell}E_{\beta,1}%
		\bigl(-t^{\beta}\psi(\eigvs)\bigr)\left(\dfco{(\rvf_0)}\dsh(\PT{x})+%
		\cfco{(\rvf_0)}\csh(\PT{x})\right),
		\label{eq:polsol1}
	\end{equation}
	which is convergent in $\vLppsph{2}{2}$.
\end{theorem}

\begin{remark}
	Theorem~\ref{thm:posol} is an extension of \cite[Theorem~1]{DOLeOr2016} to
	the vector case of fractional stochastic Cauchy problem on $\sph{2}$.
\end{remark}

\begin{remark}
	An equation of the form \eqref{eq:sCau.pb} on $\mathbb{R}^{n}$ was
	investigated in \cite{anh2001spectral}. Eq.~\eqref{eq:vfSDE.S2} is an
	(infinite-dimensional) spatiotemporal version of the 2-term fractional
	differential equation investigated in \cite{anh2002dynamic}. There, the
	roles of the fractional derivatives were studied more generally for
	multi-term time-fractional differential equations. Especially, for the
	2-term equation such as \eqref{eq:vfSDE.S2}, the order $\beta $ signifies
	the extent of intermittency in its solution. Apart from the fractional
	diffusion operator and the long-range dependence depicted by the Hurst index
	$H$ in the range $1/2<H<1$ in Eq.~\eqref{eq:vfSDE.S2}, this intermittency is
	another important feature which characterizes the temporal evolution of the
	solution.
\end{remark}

\begin{proof}[Proof of Theorem~\ref{thm:posol}]
	The proof is similar to Theorem~\ref{thm:sol.vfSDE.S2}. By \cite%
	{Djrbashian1993}, for $\beta \in (0,1]$, the solution to the fractional
	equation
	\begin{equation*}
		\frac{\mathrm{d}^{\beta}f(t)}{\mathrm{d}{t}^{\beta}}+\lambda f(t)=0
	\end{equation*}
	is $E_{\beta,1}(-\lambda t^{\beta})$, $t\geq 0$. By this result, taking the
	fractional-in-time derivatives on both sides of~\eqref{eq:polsol1} one
	obtains
	\begin{align*}
		&\pdt{t}\posol(t,\PT{x})  =\sum_{\ell
			=1}^{\infty}\sum_{m=-\ell}^{\ell}\left(-\psi(\eigvs)E_{\beta,1}(-t^{\beta}%
		\psi(\eigvs))\right)\left(\dfco{(\rvf_0)}\dsh(\PT{x})+\cfco{(\rvf_0)}\csh(%
		\PT{x})\right) \\
		&=\sum_{\ell =1}^{\infty}\sum_{m=-\ell}^{\ell}E_{\beta,1}(-t^{\beta}\psi(%
		\eigvs))\left(\dfco{(\rvf_0)}(-\psi(-\LBo))\dsh(\PT{x})+\cfco{(\rvf_0)}%
		(-\psi(-\LBo))\csh(\PT{x})\right) \\
		& =-\psi (-\LBo)\posol(t,\PT{x}),
	\end{align*}
	which verifies that $\posol(t,\PT{x})$ in \eqref{eq:polsol1} satisfies the
	equation \eqref{eq:sCau.pb}.

	By \eqref{eq:vf.Parseval} and \eqref{eq:rvf.L2},
	\begin{align}
		\norm{\posol(t,\PT{x})}{\vLppsph{2}{2}}^{2}& =\sum_{\ell
			=1}^{\infty}\sum_{m=-\ell}^{\ell}\bigl|E_{\beta,1}\bigl(-t^{\beta}\psi (%
		\eigvs)\bigr)\bigr|^{2}\expect{\left(\bigl|\dfco{(\rvf_0)}\bigr|^{2}
			+\bigl|\cfco{(\rvf_0)}\bigr|^{2}\right)}  \notag \\
		& \leq \expect{\sum_{\ell=1}^{\infty}\sum_{m=-\ell}^{\ell}\left(\bigl|%
			\dfco{(\rvf_0)}\bigr|^{2}+\bigl|\cfco{(\rvf_0)}\bigr|^{2}\right)} = %
		\norm{\rvf_0}{\vLppsph{2}{2}}^{2}<\infty,  \label{norm_u}
	\end{align}
	where it was used that $0<E_{\beta,1}(-z)\le 1$ for $z\geq 0$ (see \cite[%
	Theorem~4]{Simon2014}). This shows that the expansion of $\posol(t,\PT{x})$
	in \eqref{eq:polsol1} converges in $\vLppsph{2}{2}$. Thus, \eqref{eq:polsol1}
	is the solution to \eqref{eq:sCau.pb}.\hfill {}
\end{proof}

Let us consider the approximation $\posol_L(t,\PT{x})$ of truncation degree $%
L\in\mathbb{N}$ to the solution $\posol(t,\PT{x})$ in~\eqref{eq:polsol1}
defined by
\begin{equation*}
	\posol_L(t,\PT{x})=\sum_{\ell=1}^{L}\sum_{m=-\ell}^{\ell}E_{\beta,1}\bigl(%
	-t^{\beta}\psi(\eigvs)\bigr)\left(\dfco{(\rvf_0)}\dsh(\PT{x})+\cfco{(\rvf_0)}%
	\csh(\PT{x})\right), \;\; t\geq 0,\PT{x}\in \sph{2}.
\end{equation*}
\newline
\indent The following result gives the convergence rate of the approximation
$\posol_L(t,\PT{x})$ to the solution $\posol(t,\PT{x})$ in terms of the
divergence-free and curl-free power spectra.

\begin{corollary}
	Let the conditions of Theorem~\ref{thm:posol} be satisfied and $\posol(t,%
	\PT{x})$ be the solution to the equation \eqref{eq:polsol1}. Then, we have
	the following estimates.

	\begin{enumerate}
		\item[(i)] For $t>0$, the truncation error is bounded by
		\begin{equation*}
			\|\posol(t,\PT{x})-\posol_L(t,\PT{x})\|_{\vLppsph{2}{2}} \le
			C\left(\sum_{\ell =L+1}^{\infty}(2\ell +1)\left(\widehat{\sigma}_{\ell}^{2}+%
			\widetilde{\sigma}_{\ell}^{2}\right)\right)^{1/2}.
		\end{equation*}

		\item[(ii)] For the magnitudes of variances that decay algebraically with
		order $\nu >2$, that is, there exist constants $C >0$ and $\ell_0\in \mathbb{%
			N}$ such that $\left(\widehat{\sigma}_{\ell}^{2}+\widetilde{\sigma}%
		_{\ell}^{2}\right)\le C\cdot \ell^{-\nu}$ for all $\ell\ge \ell_0$, it holds that
		\begin{equation*}
			\|\posol(t,\PT{x})-\posol_L(t,\PT{x})\|_{\vLppsph{2}{2}} \le C {L}^{-\frac{%
					\nu-2}{2}}.
		\end{equation*}

		\item[(iii)] The estimates in (i) and (ii) are also valid uniformly over $%
		\PT{x}\in \sph{2}$ for the root mean square truncation error
		\begin{align*}
			\mathbf{RMSE}(\posol(t,\PT{x})-\posol_L(t,\PT{x})) &=\mathbf{Var}^{1/2}\left(%
			\posol(t,\PT{x})-\posol_L(t,\PT{x})\right) \\
			&=\|\posol(t,\PT{x})-\posol_L(t,\PT{x})\|_{L_2(\Omega)}.
		\end{align*}

		\item[(iv)] Under the conditions of the algebraic decay in (ii), for any $%
		\varepsilon>0$, it holds that
		\begin{equation*}
			\mathbf{P}\Big(|\posol(t,\PT{x})-\posol_L(t,\PT{x})|\ge\varepsilon\Big)\le
			\frac{C}{{L}^{\nu-2}\,\varepsilon^2}.
		\end{equation*}
	\end{enumerate}
\end{corollary}

\begin{proof}
	The proof uses the representation
	\[
		\posol(t,\PT{x})-\posol_L(t,\PT{x})=\sum_{\ell=
			L+1}^{\infty}\sum_{m=-\ell}^{\ell}E_{\beta,1}\bigl(-t^{\beta}\psi(\eigvs)%
		\bigr)\left(\dfco{(\rvf_0)}\dsh(\PT{x})+\cfco{(\rvf_0)}\csh(\PT{x})\right),
	\]
	the estimate (\ref{norm_u}), the property (\ref{norm2}), and is analogous to
	the proof of Corollary~\ref{cor1}.\hfill {}
\end{proof}

If the Fourier coefficients $\dfco{(\rvf_0)}$ and $\cfco{(\rvf_0)}$ are
uncorrelated with $\BMa(t)$ and $\BMb(t)$, then Theorems~\ref%
{thm:sol.vfSDE.S2} and \ref{thm:posol} yield the solution of %
\eqref{eq:vfSDE.S2} with initial condition $\sol(0,\PT{x})=\posol(t_{0}, %
\PT{x})$.

\begin{corollary}
	\label{cor:sol.fSDE.sCaupb} Let the conditions of Theorems~\ref%
	{thm:sol.vfSDE.S2} and \ref{thm:posol} be satisfied and $\rvf_{0}(\PT{x})$
	is uncorrelated with $\vfBmsph(t,\PT{x})$. Then, the solution of the SPDE~%
	\eqref{eq:vfSDE.S2} with the initial condition $\sol(0,\PT{x})=\posol(t_0, %
	\PT{x})$, $t_0>0$, is given by \eqref{eq:sol.vfSDE.sCaupb}.
\end{corollary}

\subsection{Covariance Function}

This subsection derives an explicit formula for the covariance matrix
function of the solution \eqref{eq:sol.vfSDE.S2} in terms of the Legendre
tensor kernels.

For $\beta \in (0,1]$, $\hurst\in \lbrack 1/2,1)$, let us define
\begin{equation}
    E_{\beta ,\hurst}^{\ast }\bigl(t,z\bigr):=\Gamma (2\hurst+1)\int_{0}^{t}u^{2%
        \beta +2\hurst-3}E_{\beta ,\beta }(-u^{\beta }z)E_{\beta ,\beta +2\hurst-1}%
    \bigl(-u^{\beta }z\bigr)\IntD{u},\;\;t,z\in \R.  \label{eq:Estar}
\end{equation}%
We will need the following convolution property of the generalized
Mittag-Leffler function (\cite[Eq. (2.2.14)]{MaHa2008}):%
\begin{equation}
    \int_{0}^{1}z^{\beta -1}\left( 1-z\right) ^{\sigma -1}E_{\alpha ,\beta
    }\left( xz^{\alpha }\right) dz=\Gamma \left( \sigma \right) E_{\alpha
        ,\sigma +\beta }\left( x\right) ,  \label{4.15}
\end{equation}%
where $\alpha >0,\ \beta ,\sigma \in \mathbb{C},\ \Re\left( \beta
\right) >0,\ \Re\left( \sigma \right) >0$.

\begin{proposition}
    \label{prop:covfn.sol2} Let the conditions of Theorems~\ref{thm:sol.vfSDE.S2}
    and \ref{thm:posol} be satisfied and the centred random tangent field $\rvf%
    _{0}(\xi,\PT{x})$ on $\sph{2}$ have normally distributed Fourier
    coefficients $\dfco{(\rvf_0)}$ and $\cfco{(\rvf_0)}$ with mean zero and
    variances $\widehat{\sigma}_{\ell}^{2}$ and $\widetilde{\sigma}_{\ell}^{2}$
    respectively. Then, the covariance matrix function of the solution $\sol(t,%
    \PT{x})$ in \eqref{eq:sol.vfSDE.sCaupb} is a tensor field on $\sph{2}$ given
    by
    \begin{align*}
        \cov{\sol(t,\PT{x}),\sol(t,\PT{y})}& =\sum_{\ell =1}^{\infty}\left\{\dlg(%
        \PT{x},\PT{y})\Bigl(\widehat{\sigma}_{\ell}^{2}\left(E_{\beta,1}\bigl(%
        -t_{0}^{\beta}\psi(\eigvs)\bigr)\right)^{2}+A_{\ell}^{1}E_{\beta,\hurst%
        }^{\ast}\bigl(t,\psi (\eigvs)\bigr)\Bigr)\right. \\
        & \hspace{1.cm}\left. +\clg(\PT{x},\PT{y})\Bigl(\widetilde{\sigma}%
        _{\ell}^{2}\left(E_{\beta,1}\bigl(-t_{0}^{\beta}\psi(\eigvs)\bigr)%
        \right)^{2}+A_{\ell}^{2}E_{\beta,\hurst}^{\ast}\bigl(t,\psi (\eigvs)\bigr)%
        \Bigr)\right\},
    \end{align*}
    where $\dlg(\PT{x},\PT{y})$ and $\clg(\PT{x},\PT{y})$ are the Legendre
    tensor kernels given by \eqref{eq:dlg.clg}.
\end{proposition}

\begin{proof}
    By assumption, $\expect{\sol(t,\PT{x})}=0$ for $t\geq 0$ and $\PT{x}\in %
    \sph{2}$. Then, by \eqref{eq:sol.vfSDE.sCaupb} and the uncorrelatedness of $%
    \dfco{(\rvf_0)},\cfco{(\rvf_0)},\BMa(t)$ and $\BMb(t)$ for $\ell \geq 1$ and
    $m=-\ell ,\dots ,\ell $,
   {\small  \begin{align}
        & \cov{\sol(t,\PT{x}),\sol(t,\PT{y})}=\expect{\sol(t,\PT{x})\otimes\sol(t,
            \PT{y})}  \notag \\
        & \hspace{0.2cm}=\mathbb{E}\Biggl[\sum_{\ell =1}^{\infty }\sum_{m=-\ell
        }^{\ell }\sum_{\ell ^{\prime }=1}^{\infty }\sum_{m^{\prime }=-\ell ^{\prime
        }}^{\ell ^{\prime }}\left\{ \left( \dfco{(\rvf_0)}E_{\beta ,1}\bigl(%
        -t_{0}^{\beta }\psi (\eigvs)\bigr)+\int_{0}^{t}\left( t-s\right) ^{\beta
            -1}E_{\beta ,\beta }\bigl(-\left( t-s\right) ^{\beta }\psi (\eigvs)\bigr)\right. \right. \notag \\
       &\hspace{0.3cm} \left. \left. \times\IntBa(s)\right) \dsh(\PT{x})
        +\left( \cfco{(\rvf_0)}E_{\beta ,1}\bigl(%
        -t_{0}^{\beta }\psi (\eigvs)\bigr)+\int_{0}^{t}\left( t-s\right) ^{\beta
            -1}E_{\beta ,\beta }\bigl(-\left( t-s\right) ^{\beta }\psi (\eigvs)\bigr)\right. \right. \notag \\
                   &\hspace{0.3cm} \left. \left. \times
        \IntBb(s)\right) \csh(\PT{x})\right\}  \otimes \left\{ \left( \dfco[\ell' m']{(\rvf_0)}E_{\beta ,1}%
        \bigl(-t_{0}^{\beta }\psi (\eigvs[\ell'])\bigr)+\int_{0}^{t}\left(
        t-s\right) ^{\beta -1}E_{\beta ,\beta }\bigl(-\left( t-s\right) ^{\beta
        }\psi (\eigvs[\ell'])\bigr)
        \right. \right. \notag
                 \end{align}
    }
{\small  \begin{align}
        & \left. \left. \times \IntBa[\ell' m'](s)\right) \dsh[\ell' m'](\PT{y}%
        +\left( \cfco[\ell' m']{(\rvf_0)}E_{\beta ,1}\bigl(%
        -t_{0}^{\beta }\psi (\eigvs[\ell'])\bigr)+\int_{0}^{t}\left( t-s\right)
        ^{\beta -1}E_{\beta ,\beta }\bigl(-\left( t-s\right) ^{\beta }\psi (%
        \eigvs[\ell'])\bigr) \right. \right. \notag \\
        & \left. \left. \times \IntBb[\ell' m'](s)\right) \csh[\ell' m'](\PT{y}%
        )\right\} \Biggr]  =\mathbb{E}\Biggl[\sum_{\ell =1}^{\infty }\sum_{m=-\ell
        }^{\ell }\biggl\{\dsh(\PT{x})\otimes \dsh(\PT{y})\biggl(\bigl|\dfco{(\rvf_0)}%
        E_{\beta ,1}\bigl(-t_{0}^{\beta }\psi (\eigvs)\bigr)\bigr|^{2}  \notag \\
        & \hspace{0.1cm} +\int_{0}^{t}\int_{0}^{t}\left( t-u\right) ^{\beta -1}\left(
        t-v\right) ^{\beta -1}E_{\beta ,\beta }\bigl(-\left( t-u\right) ^{\beta
        }\psi (\eigvs)\bigr)E_{\beta ,\beta }\bigl(-\left( t-v\right) ^{\beta }\psi (%
        \eigvs)\bigr) \IntBa(u)\notag \\
        &\hspace{0.1cm} \times \IntBa(v)\biggr) +\csh(\PT{x})\otimes \csh(\PT{y})\biggl(\bigl|\cfco{(\rvf_0)}%
        E_{\beta ,1}\bigl(-t_{0}^{\beta }\psi (\eigvs)\bigr)\bigr|^{2}+\int_{0}^{t}\int_{0}^{t}\left( t-u\right) ^{\beta -1} \notag \\
        &\hspace{1.5cm} \times \left(
        t-v\right) ^{\beta -1} E_{\beta ,\beta }\bigl(-\left( t-u\right) ^{\beta
        }\psi (\eigvs)\bigr)E_{\beta ,\beta }\bigl(-\left( t-v\right) ^{\beta }\psi (%
        \eigvs)\bigr)\IntBb(u)\IntBb(v)\biggr)\biggr\}\Biggr],  \notag \\[-8mm]
        &  \label{eq:cov.sol1}
    \end{align}}
where the first equality uses that $\dfco{(\rvf_0)},\cfco{(\rvf_0)},\BMa(t)$
    and $\BMb(t)$ are centred and the third equality uses their uncorrelatedness.

   By \cite[Eq~1.3]{MeMiVa2001}, the expectation of each double stochastic
    integral in \eqref{eq:cov.sol1} becomes
    {\small \begin{align}
        & \mathbb{E}\left( \int_{0}^{t}\int_{0}^{t}\left( t-u\right) ^{\beta
            -1}\left( t-v\right) ^{\beta -1}E_{\beta ,\beta }\bigl(-\left( t-u\right)
        ^{\beta }\psi (\eigvs)\bigr)E_{\beta ,\beta }\bigl(-\left( t-v\right)
        ^{\beta }\psi (\eigvs)\bigr)\IntBa[\ell' m'](u)\right.\notag \\
        &\qquad \left. \times\IntBb\left( v\right) \right)
=A_{\ell }^{1}\hurst(2\hurst-1)\int_{0}^{t}\int_{0}^{t}%
        \left( t-u\right) ^{\beta -1}\left( t-v\right) ^{\beta -1}E_{\beta ,\beta
        }(-\left( t-u\right) ^{\beta }\psi (\eigvs))\notag \\
        &\qquad \times E_{\beta ,\beta }(-\left(
        t-v\right) ^{\beta }\psi (\eigvs))|u-v|^{2\hurst-2}\IntD{u}\IntD{v} =A_{\ell }^{1}\hurst(2\hurst-1)\int_{0}^{t}\left( t-u\right)
        ^{\beta -1}\notag \\
        &\qquad  \times E_{\beta ,\beta }(-\left( t-u\right) ^{\beta }\psi (\eigvs))%
        \IntD{u}   \left( \int_{0}^{u}\left( t-v\right) ^{\beta
            -1}E_{\beta ,\beta }(-\left( t-v\right) ^{\beta }\psi (\eigvs))(u-v)^{2\hurst
            -2}\IntD{v}\right.\notag \\
            &\qquad \left. +\int_{u}^{t}\left( t-v\right) ^{\beta -1}E_{\beta ,\beta
        }(-\left( t-v\right) ^{\beta }\psi (\eigvs))(v-u)^{2\hurst-2}\IntD{v}\right). \notag \\[-7mm]
        &  \label{eq:cov.sInt}
    \end{align}\vspace{2mm}
}

    The first integral in the large brackets reads%
    \begin{equation*}
        \int_{0}^{u}\left( t-v\right) ^{\beta -1}E_{\beta ,\beta }(-\left(
        t-v\right) ^{\beta }\psi (\lambda _{l}))(t-v-\left( t-u\right) )^{2H-2}dv
    \end{equation*}%
    \begin{equation*}
        =\int_{t-u}^{t}z^{\beta -1}E_{\beta ,\beta }(-z^{\beta }\psi (\lambda
        _{l}))\left( z-(t-u)\right) ^{2H-2}dz
    \end{equation*}\vspace{2mm}
    by the change of variable $t-v=z$. Thus,%
    {\small \begin{align}
   &     \int_{0}^{t}\left( t-u\right) ^{\beta -1}E_{\beta ,\beta }(-\left(
        t-u\right) ^{\beta }\psi (\lambda _{l}))\left( \int_{0}^{u}\left( t-v\right)
        ^{\beta -1}E_{\beta ,\beta }(-\left( t-v\right) ^{\beta }\psi (\lambda
        _{l}))(u-v)^{2H-2}dv\right) du\notag \\
& \hspace{2mm}       =\int_{0}^{t}\left( t-u\right) ^{\beta -1}E_{\beta ,\beta }(-\left(
        t-u\right) ^{\beta }\psi (\lambda _{l}))\left( \int_{t-u}^{t}z^{\beta
            -1}E_{\beta ,\beta }(-z^{\beta }\psi (\lambda _{l}))\left( z-(t-u)\right)
        ^{2H-2}dz\right) du\notag \\
&\hspace{2mm}      =\int_{0}^{t}s^{\beta -1}E_{\beta ,\beta }(-s^{\beta }\psi (\lambda
        _{l}))\left( \int_{s}^{t}z^{\beta -1}E_{\beta ,\beta }(-z^{\beta }\psi
        (\lambda _{l}))\left( z-s\right) ^{2H-2}dz\right) ds\notag \\
& \qquad\qquad       =\int_{0}^{t}z^{\beta -1}E_{\beta ,\beta }(-z^{\beta }\psi (\lambda
        _{l}))\left( \int_{0}^{z}s^{\beta -1}E_{\beta ,\beta }(-s^{\beta }\psi
        (\lambda _{l}))\left( z-s\right) ^{2H-2}ds\right) dz,\notag \\[-7mm]
        &   \label{4.16}
    \end{align}}
    by the change of
    variable $t-u=s$ and changing the order of integration.

    For the second integral in the large
    brackets, with $v-u=t-u-\left( t-v\right) $ and the change of variable $t-v=s
    $, we get%
    \begin{equation*}
        \int_{u}^{t}\left( t-v\right) ^{\beta -1}E_{\beta ,\beta }(-\left(
        t-v\right) ^{\beta }\psi (\lambda _{l}))(v-u)^{2H-2}dv
    \end{equation*}%
    \begin{equation*}
        =\int_{0}^{t-u}s^{\beta -1}E_{\beta ,\beta }(-s^{\beta }\psi (\lambda
        _{l}))(t-u-s)^{2H-2}ds.
    \end{equation*}%
    Thus,
        {\small
    \begin{equation*}
        \int_{0}^{t}\left( t-u\right) ^{\beta -1}E_{\beta ,\beta }(-\left(
        t-u\right) ^{\beta }\psi (\lambda _{l}))\left( \int_{u}^{t}\left( t-v\right)
        ^{\beta -1}E_{\beta ,\beta }(-\left( t-v\right) ^{\beta }\psi (\lambda
        _{l}))(v-u)^{2H-2}dv\right) du
    \end{equation*}%
    \begin{equation*}
        =\int_{0}^{t}\left( t-u\right) ^{\beta -1}E_{\beta ,\beta }(-\left(
        t-u\right) ^{\beta }\psi (\lambda _{l}))\left( \int_{0}^{t-u}s^{\beta
            -1}E_{\beta ,\beta }(-s^{\beta }\psi (\lambda _{l}))(t-u-s)^{2H-2}ds\right)
        du
    \end{equation*}%
    \begin{equation}
        =\int_{0}^{t}z^{\beta -1}E_{\beta ,\beta }(-z^{\beta }\psi (\lambda
        _{l}))\left( \int_{0}^{z}s^{\beta -1}E_{\beta ,\beta }(-s^{\beta }\psi
        (\lambda _{l}))\left( z-s\right) ^{2H-2}ds\right) dz  \label{4.17}
    \end{equation}\\
    \vspace{2mm}}
    by the change of variable $t-u=z$. In view of (\ref{4.16}), (\ref{4.17}) and
    (\ref{eq:cov.sInt}), we then obtain%
      {\small   \begin{equation*}
        \mathbb{E}\left( \int_{0}^{t}\int_{0}^{t}\left( t-u\right) ^{\beta -1}\left(
        t-v\right) ^{\beta -1}E_{\beta ,\beta }\bigl(-\left( t-u\right) ^{\beta
        }\psi (\eigvs)\bigr)E_{\beta ,\beta }\bigl(-\left( t-v\right) ^{\beta }\psi (%
        \eigvs)\bigr)\IntBa[\ell' m'](u)\IntBb\left( v\right) \right)
    \end{equation*}%
    \begin{align}
   &\hspace{1.2cm}     =2A_{l}^{1}H\left( 2H-1\right) \int_{0}^{t}z^{\beta -1}E_{\beta ,\beta
        }(-z^{\beta }\psi (\lambda _{l}))\left( \int_{0}^{z}s^{\beta -1}E_{\beta
            ,\beta }(-s^{\beta }\psi (\lambda _{l}))\left( z-s\right) ^{2H-2}ds\right)
        dz.
\notag \\[-7mm]
&   \label{4.18}
\end{align}}
    With a rescaling $w=\frac{s}{z}$, we get%
    \begin{equation*}
        \int_{0}^{z}s^{\beta -1}E_{\beta ,\beta }(-s^{\beta }\psi (\lambda
        _{l}))\left( z-s\right) ^{2H-2}ds
    \end{equation*}%
    \begin{equation*}
        =z^{\beta +2H-2}\int_{0}^{1}w^{\beta -1}\left( 1-w\right) ^{2H-2}E_{\beta
            ,\beta }(-z^{\beta }\psi (\lambda _{l})w^{\beta })dw
    \end{equation*}%
    \begin{equation}
        =z^{\beta +2H-2}\Gamma \left( 2H-1\right) E_{\beta ,\beta +2H-1}(-z^{\beta
        }\psi (\lambda _{l}))  \label{4.19}
    \end{equation}%
    using the convolution property (\ref{4.15}) of the generalized
    Mittag-Leffler function. Putting (\ref{4.19}) into (\ref{4.18}) we finally
    obtain%
{\small   \begin{equation*}
        \mathbb{E}\left( \int_{0}^{t}\int_{0}^{t}\left( t-u\right) ^{\beta -1}\left(
        t-v\right) ^{\beta -1}E_{\beta ,\beta }\bigl(-\left( t-u\right) ^{\beta
        }\psi (\eigvs)\bigr)E_{\beta ,\beta }\bigl(-\left( t-v\right) ^{\beta }\psi (%
        \eigvs)\bigr)\IntBa[\ell' m'](u)\IntBb\left( v\right) \right)
    \end{equation*}%
    \begin{equation*}
        =A_{l}^{1}\Gamma \left( 2H+1\right) \int_{0}^{t}z^{2\beta +2H-3}E_{\beta
            ,\beta }(-z^{\beta }\psi (\lambda _{l}))E_{\beta ,\beta +2H-1}(-z^{\beta
        }\psi (\lambda _{l}))dz,
    \end{equation*}%
}
    which is the expression (\ref{eq:Estar}) with an according change of
    notation.

    Using the notation of \eqref{eq:Estar}, it follows from \eqref{eq:cov.sol1}
    and \eqref{eq:cov.sInt} that
    \begin{align}
        & \cov{\sol(t,\PT{x}),\sol(t,\PT{y})}  \notag \\
        & =\sum_{\ell =1}^{\infty }\sum_{m=-\ell }^{\ell }\Bigl\{\dsh(\PT{x})\otimes %
        \dsh(\PT{y})\Bigl(\expect{\bigl|\dfco{(\rvf_0)}\bigr|^{2}}\left( E_{\beta ,1}%
        \bigl(-t_{0}^{\beta }\psi (\eigvs)\bigr)\right) ^{2}+A_{\ell }^{1}E_{\beta ,%
            \hurst}^{\ast }\bigl(t,\psi (\eigvs)\bigr)\Bigl)  \notag \\
        & \qquad  +\csh(\PT{x})\otimes \csh(\PT{y})\Bigl(\expect{\bigl|\cfco{(%
                \rvf_0)}\bigr|^{2}}\left( E_{\beta ,1}\bigl(-t_{0}^{\beta }\psi (\eigvs)%
        \bigr)\right) ^{2}+A_{\ell }^{2}E_{\beta ,\hurst}^{\ast }\bigl(t,\psi (\eigvs%
        )\bigr)\Bigl)\Bigr\}  \notag \\
        & =\sum_{\ell =1}^{\infty }\left\{ \dlg(\PT{x},\PT{y})\Bigl(\widehat{\sigma }%
        _{\ell }^{2}\left( E_{\beta ,1}\bigl(-t_{0}^{\beta }\psi (\eigvs)\bigr)%
        \right) ^{2}+A_{\ell }^{1}E_{\beta ,\hurst}^{\ast }\bigl(t,\psi (\eigvs)%
        \bigr)\Bigr)\right.   \notag \\
        & \qquad \qquad \left. +\clg(\PT{x},\PT{y})\Bigl(\widetilde{\sigma }_{\ell
        }^{2}\left( E_{\beta ,1}\bigl(-t_{0}^{\beta }\psi (\eigvs)\bigr)\right)
        ^{2}+A_{\ell }^{2}E_{\beta ,\hurst}^{\ast }\bigl(t,\psi (\eigvs)\bigr)\Bigr)%
        \right\} ,  \notag \\[-7mm]
&   \label{estpq}
\end{align}
    where the last equality uses the addition theorem \eqref{eq:add.thm.vsh} for
    vector spherical harmonics.\hfill {}
\end{proof}

\begin{corollary}
	Let the conditions of Proposition~\ref{prop:covfn.sol2} be satisfied. Then,
	we have the following estimates.

	\begin{enumerate}
		\item[(i)] For $t>0$ the truncation error is bounded by
		\begin{equation*}
			\|\sol(t,\PT{x})-\sol_L(t,\PT{x})\|_{\vLppsph{2}{2}} \le
			C\left(\sum_{\ell=L+1}^{\infty}(2\ell+1)\left(\widehat{\sigma}_{\ell}^{2}+%
			\widetilde{\sigma}_{\ell}^{2}+\left(A_{\ell}^{1}+A_{\ell}^{2}\right)\left(%
			\psi(\eigvs)\right)^{\tau}\right)\right)^{1/2}.
		\end{equation*}

		\item[(ii)] For the magnitudes of variances that decay algebraically with
		order $\nu >2$, that is, there exist constants $C >0$ and $\ell_0\in \mathbb{%
			N}$ such that $\widehat{\sigma}_{\ell}^{2}+\widetilde{\sigma}%
		_{\ell}^{2}+\left(A_{\ell}^{1}+A_{\ell}^{2}\right)\left(\psi(\eigvs%
		)\right)^{\tau}\le C\cdot \ell^{-\nu}$ for all $\ell\ge \ell_0$, it holds
		\begin{equation*}
			\|\sol(t,\PT{x})-\sol_L(t,\PT{x})\|_{\vLppsph{2}{2}} \le C {L}^{-\frac{\nu-2%
				}{2}}.
		\end{equation*}

		\item[(iii)] Under conditions of the algebraic decay in (ii), for any $%
		\varepsilon>0$ it holds that
		\begin{equation*}
			\mathbf{P}\Big(|\sol(t,\PT{x})-\sol_L(t,\PT{x})|\ge\varepsilon\Big) \le
			\frac{C}{{L}^{\nu-2}\,\varepsilon^2}.
		\end{equation*}
	\end{enumerate}
\end{corollary}

\begin{proof}
	One needs to use the following modification of the result in (\ref{estpq}):
	\begin{align}
		& \cov{\sol(t,\PT{x})-\sol_L(t,\PT{x}),\sol(t,\PT{x})-\sol_L(t,\PT{x})}
		\notag \\
		& =\sum_{\ell =L+1}^{\infty }\left\{ \dlg(\PT{x},\PT{y})\Bigl(\widehat{%
			\sigma }_{\ell }^{2}\left( E_{\beta ,1}\bigl(-t_{0}^{\beta }\psi (\eigvs)%
		\bigr)\right) ^{2}+A_{\ell }^{1}E_{\beta ,\hurst}^{\ast }\bigl(t,\psi (\eigvs%
		)\bigr)\Bigr)\right.  \notag \\
		& \qquad \qquad \left. +\clg(\PT{x},\PT{y})\Bigl(\widetilde{\sigma }_{\ell
		}^{2}\left( E_{\beta ,1}\bigl(-t_{0}^{\beta }\psi (\eigvs)\bigr)\right)
		^{2}+A_{\ell }^{2}E_{\beta ,\hurst}^{\ast }\bigl(t,\psi (\eigvs)\bigr)\Bigr)%
		\right\} .  \notag
	\end{align}%
	The statements of the Corollary then follow from the estimates of the
	Mittag-Leffler function for $z\geq 0:$
	\begin{equation*}
		0<E_{\beta ,1}(-z)\leq 1,
	\end{equation*}%
	\begin{align}
		& E_{\beta ,\hurst}^{\ast }\bigl(t,z\bigr)=\frac{\Gamma (2\hurst+1)}{z^{%
				\frac{2}{\beta }(\beta +H-1)}}\int_{0}^{tz^{1/\beta }}u^{2\beta +2\hurst%
			-3}E_{\beta ,\beta }(-u^{\beta })E_{\beta ,\beta +2\hurst-1}\bigl(-u^{\beta }%
		\bigr)\IntD{u}  \notag \\
		& \quad \leq \frac{\Gamma (2\hurst+1)}{z^{\frac{2}{\beta }(\beta +H-1)}}%
		\left( \int_{0}^{+\infty }\left( u^{\beta -1}E_{\beta ,\beta }(-u^{\beta
		})\right) ^{2}\IntD{u}\int_{0}^{+\infty }\left( u^{\beta +2\hurst-2}E_{\beta
			,\beta +2\hurst-1}\bigl(-u^{\beta }\bigr)\right) ^{2}\IntD{u}\right) ^{1/2}
		\notag \\
		& \quad \leq Cz^{\frac{2}{\beta }(1-\beta -H)}.  \notag
	\end{align}%
	and the upper bounds for the norms of the divergence-free and curl-free
	Legendre tensor kernels
	\begin{equation*}
		\Vert \clg(\PT{x},\PT{y})\Vert _{\vLppsph{2}{2}}\leq C(2\ell +1),\quad \Vert %
		\dlg(\PT{x},\PT{y})\Vert _{\vLppsph{2}{2}}\leq C(2\ell +1).
	\end{equation*}%
	Then the proof is similar to the proof of Corollary~\ref{cor1}.\hfill
	{}\hfill {}
\end{proof}

\section{Analysis of temporal increments}

This section derives an upper bound in the $L_{2}\left( \Omega \times
\mathbb{S}^{2}\right) $-norm of the temporal increments of the solution to
Equation (\ref{eq:vfSDE.S2}). The result shows the interplay between the
exponent $\beta $ of the fractional derivative in time and the Hurst
parameter $H$.

\begin{proposition}
	\label{Var} Let the conditions of Theorem~\ref{thm:sol.vfSDE.S2} be
	satisfied and $1-H<\beta <1.$ Then, for any $t, \tau\ge 0,$ which are not
	simultaneously equal to zero, the norm of the temporal increments of the
	solution $\mathbf{X}\left( t,\mathbf{x}\right) $ to Equation (\ref%
	{eq:vfSDE.S2}) is bounded as
	\begin{equation*}
		\left\Vert \mathbf{X}\left( t,\mathbf{x}\right) -\mathbf{X}\left( \tau ,%
		\mathbf{x}\right) \right\Vert _{L_{2}\left( \Omega \times \mathbb{S}%
			^{2}\right) }^{2}\leq C \frac{\left| t-\tau\right|^{2\hurst}}{%
			(t+\tau)^{2(1-\beta)}},
	\end{equation*}%
	where $C$ is a positive constant.
\end{proposition}

\begin{proof}
    Due to the symmetry, we consider only the case $t>\tau \geq 0$. Using the
    representation~\eqref{eq:sol.vfSDE.S2} of the solution $\mathbf{X}\left( t,%
    \mathbf{x}\right) $ and repeating first steps in the proof of %
    \eqref{proofTh4.1} one obtains
  {\small  \begin{align}
        & \bigl\|\mathbf{X}\left( t,\mathbf{x}\right) -\mathbf{X}\left( \tau ,%
        \mathbf{x}\right) \bigr\|_{L_{2}\left( \Omega \times \mathbb{S}^{2}\right)
        }^{2}  \notag \\
        & =\sum_{\ell =1}^{\infty }\sum_{m=-\ell }^{\ell }\left\{ \mathbb{E}\left[
        \left\vert \int_{\tau }^{t}\left( t-s\right) ^{\beta -1}E_{\beta ,\beta
        }(-\left( t-s\right) ^{\beta }\psi (\eigvs))\IntBa(s)\dsh(\PT{x})\right\vert
        ^{2}\right] \right.  \notag \\
        & \hspace{1.5cm}\left. +\mathbb{E}\left[ \left\vert \int_{\tau }^{t}\left(
        t-s\right) ^{\beta -1}E_{\beta ,\beta }(-\left( t-s\right) ^{\beta }\psi (%
        \eigvs))\IntBb(s)\dsh(\PT{x})\right\vert ^{2}\right] \right\}  \notag \\
        & \qquad \quad \leq \sum_{\ell =1}^{\infty }(2\ell +1)\left( C_{\hurst%
        }^{1}A_{\ell }^{1}+C_{\hurst}^{2}A_{\ell }^{2}\right) \left( \psi (\eigvs%
        )\right) ^{\frac{2}{\beta }(1-\beta -\hurst)}\left( \int_{\tau \left( \psi (%
            \eigvs)\right) ^{\frac{1}{\beta }}}^{t\left( \psi (\eigvs)\right) ^{\frac{1}{%
                    \beta }}}|u^{\beta -1}E_{\beta ,\beta }(-u^{\beta })|^{\frac{1}{\hurst}}%
        \IntD{u}\right) ^{2\hurst}\hspace{-3mm}.      \notag \\[-8mm]
        &   \label{pr1}
    \end{align}}

    \vspace{2mm} By the inequality
    \begin{equation}
        |E_{\alpha,\beta}(-z)|\le\frac{C}{1+z},\quad \alpha <2,\ \beta\in\mathbb{R}%
        ,\ z\ge 0,  \label{boundE}
    \end{equation}
    see \cite[Theorem~1.6]{Podlubny1999}, it follows that $\bigl|%
    E_{\beta,\beta}(-u^{\beta})\bigr|\le C$ and
    \begin{equation*}
        \int_{\tau\left(\psi(\eigvs)\right)^{\frac{1}{\beta}}}^{t\left(\psi(\eigvs%
            )\right)^{\frac{1}{\beta}}}|u^{\beta-1}E_{\beta,\beta}(-u^{\beta})|^{\frac{1%
            }{\hurst}}\IntD{u} \le C\int_{\tau\left(\psi(\eigvs)\right)^{\frac{1}{\beta}%
        }}^{t\left(\psi(\eigvs)\right)^{\frac{1}{\beta}}}u^{\frac{\beta-1}{\hurst}}%
        \IntD{u}     \end{equation*}
            \begin{equation*}\le C\frac{\left(\psi(\eigvs)\right)^{\frac{1}{\beta}+\frac{\beta-1%
                }{\beta\hurst}}}{1+\frac{\beta-1}{\hurst}}\left(t^{1+\frac{\beta-1}{\hurst}%
        }-\tau^{1+\frac{\beta-1}{\hurst}}\right).
    \end{equation*}
    Now, by applying the inequality, see \cite[3.6.24]{Mitrinovic1970},
    \begin{equation*}
        z^a-1\le (z+1)^{a-1}(z-1),\quad z \ge 1,\ a\in (0,1),
    \end{equation*}
    to the right-hand side above one obtains
    \begin{equation}  \label{pr2}
        \left(\int_{\tau\left( \psi (\eigvs)\right)^{\frac{1}{\beta}}}^{t\left( \psi
            (\eigvs)\right)^{\frac{1}{\beta}}}|u^{\beta -1}E_{\beta, \beta
        }(-u^{\beta})|^{\frac{1}{\hurst}}\IntD{u}\right)^{2\hurst} \le C \left(\psi (%
        \eigvs)\right)^{\frac{2}{\beta}(\beta+\hurst-1)} \frac{\left(
            t-\tau\right)^{2\hurst}}{(t+\tau)^{2(1-\beta)}}.
    \end{equation}
    Thus, by (\ref{pr1}) and (\ref{pr2}) it follows
    \begin{equation*}
        \bigl\|\mathbf{X}\left(t,\mathbf{x}\right)-\mathbf{X}\left(\tau,\mathbf{x}%
        \right) \bigr\|_{L_{2}\left(\Omega\times \mathbb{S}^{2}\right)}^{2} \le
        C\sum_{\ell=1}^{\infty}(2\ell+1) \left(C_{\hurst}^{1}A_{\ell}^{1}+C_{\hurst%
        }^{2}A_{\ell}^{2}\right)\frac{(t-\tau)^{2\hurst}}{(t+\tau)^{2(1-\beta)}},
    \end{equation*}
    which by \eqref{brspect} completes the proof.
\end{proof}

\begin{remark}
	Note that if the distance $|t-\tau|$ is bounded, then
	\begin{equation*}
		\lim_{t,\tau \to \infty}\left\Vert \mathbf{X}\left( t,\mathbf{x}\right) -%
		\mathbf{X}\left( \tau ,\mathbf{x}\right) \right\Vert _{L_{2}\left( \Omega
			\times \mathbb{S}^{2}\right) }=0.
	\end{equation*}

	If $\tau=0,$ then the deviations of the solution $\mathbf{X}\left( t,\mathbf{%
		x}\right)$ at time $t$ from the initial condition $\mathbf{X}\left( 0 ,%
	\mathbf{x}\right) $ are bounded as
	\begin{equation*}
		\left\Vert \mathbf{X}\left( t,\mathbf{x}\right) -\mathbf{X}\left( 0 ,\mathbf{%
			x}\right) \right\Vert _{L_{2}\left( \Omega \times \mathbb{S}^{2}\right)
		}^{2}\leq C t^{2(\hurst+\beta-1)}.
	\end{equation*}
	As $1-H<\beta,$ the above upper bound approaches zero when $t\to 0.$
\end{remark}

\begin{corollary}
	\label{corVar} Let the conditions of Proposition~\ref{Var} be satisfied.
	Then, for any $t, \tau> 0,$ the norm of the temporal increments of the
	solution $\mathbf{X}\left( t,\mathbf{x}\right) $ to Equation (\ref%
	{eq:vfSDE.S2}) is bounded as
	\begin{equation*}
		\left\Vert \mathbf{X}\left( t,\mathbf{x}\right) -\mathbf{X}\left( \tau ,%
		\mathbf{x}\right) \right\Vert _{L_{2}\left( \Omega \times \mathbb{S}%
			^{2}\right) }^{2}\leq C \frac{\left| t-\tau\right|^{2\hurst}}{(t+\tau)^4},
	\end{equation*}%
	where $C$ is a positive constant.
\end{corollary}

\begin{proof}
	The proof is analogous to the proof of Propositions~\ref{Var} if for $u> 0$
	one applies the estimate $|E_{\beta ,\beta }(-u^{\beta })|\le {C}/{u^{\beta}}%
	,$ which follows from (\ref{boundE}).
\end{proof}

The upper bound in Corollary~\ref{corVar} is preferable to Proposition~\ref%
{Var} for large values of $t$ and $\tau.$

\section*{Acknowledgments}

This research was partially supported under the Australian Research
Council's Discovery Projects funding scheme (project number DP160101366). Yu Guang Wang acknowledges the support of funding from the European Research
Council (ERC) under the European Union's Horizon 2020 research and
innovation programme (grant agreement n\textsuperscript{o} 757983). The authors are also grateful to the referees for their constructive comments and suggestions that helped improve the paper.
\setcounter{section}{1}

\bibliographystyle{amsplain}
\bibliography{vfSDE}

\end{document}